\documentclass[11pt]{article}
\usepackage{amsmath,amsthm,amssymb}
\usepackage[usenames,dvipsnames]{xcolor}
\usepackage{enumerate}
\usepackage{graphicx}
\usepackage{cite}
\usepackage{comment}
\usepackage{oands}
\usepackage{tikz}
\usepackage{changepage}
\usepackage{bbm}
\usepackage{mathtools}
\usepackage[margin=1in]{geometry}
\usepackage[pagewise,mathlines]{lineno}
\usepackage{appendix}
\usepackage{stmaryrd} 
\usepackage{multicol}
\usepackage{microtype}
\usepackage[colorinlistoftodos]{todonotes}

\usepackage[pdftitle={Geodesic networks in LQG surfaces},
  pdfauthor={Ewain Gwynne},
colorlinks=true,linkcolor=NavyBlue,urlcolor=RoyalBlue,citecolor=PineGreen,bookmarks=true,bookmarksopen=true,bookmarksopenlevel=2,unicode=true,linktocpage]{hyperref}

\theoremstyle{plain}
\newtheorem{thm}{Theorem}[section]
\newtheorem{cor}[thm]{Corollary}
\newtheorem{lem}[thm]{Lemma}
\newtheorem{prop}[thm]{Proposition}

\def\@rst #1 #2other{#1}
\newcommand\MR[1]{\relax\ifhmode\unskip\spacefactor3000 \space\fi
  \MRhref{\expandafter\@rst #1 other}{#1}}
\newcommand{\MRhref}[2]{\href{http://www.ams.org/mathscinet-getitem?mr=#1}{MR#2}}

\theoremstyle{definition}
\newtheorem{defn}[thm]{Definition}
\newtheorem{remark}[thm]{Remark}

\newtheorem{ques}[thm]{Question}

\numberwithin{equation}{section}

\newcommand{\dsb}{\begin{adjustwidth}{2.5em}{0pt}
\begin{footnotesize}}
\newcommand{\dse}{\end{footnotesize}
\end{adjustwidth}}

\newcommand{\ssb}{\begin{adjustwidth}{2.5em}{0pt}}
\newcommand{\sse}{\end{adjustwidth}}

\newcommand{\aryb}{\begin{eqnarray*}}
\newcommand{\arye}{\end{eqnarray*}}
\def\alb#1\ale{\begin{align*}#1\end{align*}}
\def\allb#1\alle{\begin{align}#1\end{align}}
\newcommand{\eqb}{\begin{equation}}
\newcommand{\eqe}{\end{equation}}
\newcommand{\eqbn}{\begin{equation*}}
\newcommand{\eqen}{\end{equation*}}

\newcommand{\BB}{\mathbbm}
\newcommand{\ol}{\overline}
\newcommand{\ul}{\underline}
\newcommand{\op}{\operatorname}

\newcommand{\bd}{\mathbf}

\newcommand{\frk}{\mathfrak}

\newcommand{\ep}{\varepsilon}
\newcommand{\rta}{\rightarrow}

\newcommand{\wt}{\widetilde}
 
\newcommand{\mcl}{\mathcal}

\newcommand{\bdy}{\partial}

\let\originalleft\left
\let\originalright\right
\renewcommand{\left}{\mathopen{}\mathclose\bgroup\originalleft}
\renewcommand{\right}{\aftergroup\egroup\originalright}

\title{Geodesic networks in Liouville quantum gravity surfaces}
 \date{ }
\author{Ewain Gwynne\footnote{\url{ewain@uchicago.edu}}  \\ {\it University of Chicago}}

\begin{document}

\maketitle

\begin{abstract}
Recent work has shown that for $\gamma \in (0,2)$, a Liouville quantum gravity (LQG) surface can be endowed with a canonical metric. We prove several results concerning geodesics for this metric. In particular, we completely classify the possible networks of geodesics from a typical point on the surface to an arbitrary point on the surface, as well as the types of networks of geodesics joining two points which occur for a dense set of pairs of points on the surface. This latter result is the $\gamma$-LQG analog of the classification of geodesic networks in the Brownian map due to Angel, Kolesnik, and Miermont (2017). We also show that there is a deterministic $m\in\mathbb N$ such that almost surely any two points are joined by at most $m$ distinct LQG geodesics. 
\end{abstract}

\tableofcontents

\section{Introduction}
\label{sec-intro}

\subsection{Overview}
\label{sec-overview}

Liouville quantum gravity (LQG) is a family of canonical models of random surfaces (two-dimensional Riemannian manifolds), indexed by a parameter $\gamma \in (0,2)$, which were first introduced in the physics literature by Polyakov~\cite{polyakov-qg1}. Such surfaces are conjectured to describe the scaling limits of ``discrete random surfaces", such as random planar maps. 
See~\cite{gwynne-ams-survey,bp-lqg-notes} for introductory expository articles on LQG. 

To define LQG, let $U\subset\BB C$ be an open domain, and let $h$ be a variant of the Gaussian free field (GFF) on $U$. 
Heuristically speaking, the \emph{$\gamma$-LQG surface} corresponding to $(U,h)$ is the random two-dimensional Riemannian manifold with Riemannian metric tensor $e^{\gamma h} \, (dx^2+dy^2)$, where $dx^2+dy^2$ is the Euclidean metric tensor. 
This definition does not make literal sense since $h$ is a random distribution (generalized function), so does not have well-defined pointwise values. However, it is possible to define LQG surfaces rigorously using various regularization procedures, as we now explain.

We start with a family of continuous functions which approximate $h$ in some sense. For concreteness, for $\ep > 0$ we define
\eqb \label{eqn-gff-convolve}
h_\ep^*(z) := (h*p_{\ep^2/2})(z) = \int_{U} h(w) p_{\ep^2/2} (z,w) \, d^2w ,\quad \forall z\in U  ,
\eqe
where $p_s(z,w) =   \frac{1}{2\pi s} \exp\left( - \frac{|z-w|^2}{2s} \right) $ is the heat kernel on $\BB C$ and the integral is interpreted in the sense of distributional pairing. One can then define the volume form associated with an LQG surface, a.k.a.\ the LQG area measure, as the a.s.\ weak limit
\eqb \label{eqn-lqg-measure}
\mu_h := \lim_{\ep\rta 0} \ep^{\gamma^2/2} e^{\gamma h_\ep^*(z)} \,d^2 z, 
\eqe
where $d^2 z$ denotes Lebesgue measure on $U$. This measure is a special case of Gaussian multiplicative chaos~\cite{kahane} and many of its basic properties are proven in~\cite{shef-kpz}.
 
Recently, it was shown in a series of papers~\cite{dddf-lfpp,local-metrics,lqg-metric-estimates,gm-confluence,gm-uniqueness,gm-coord-change} that one can define the Riemannian distance function, a.k.a.\ the LQG metric, via a similar regularization procedure. 
To describe this regularization procedure, we first let $d_\gamma > 2$ be the LQG dimension exponent from~\cite{dzz-heat-kernel,dg-lqg-dim}. 
This exponent describe distances in various approximations of LQG (such as random planar maps and regularized versions of the metric tensor). 
For example, for a certain class of infinite-volume random planar maps which are expected to have LQG as their scaling limit, the number of vertices in the graph-distance ball of radius $r$ centered at the root vertex grows like $r^{d_\gamma+o(1)}$ as $r\rta\infty$~\cite[Theorem 1.6]{dg-lqg-dim}.
Once the LQG metric $D_h$ has been defined, it can be shown that $d_\gamma$ is the Hausdorff dimension of the metric space $(U,D_h)$~\cite[Corollary 1.7]{gp-kpz}. 

To construct the LQG metric, we let $h_\ep^*$ be as in~\eqref{eqn-gff-convolve} and define a metric on $U$ by
\eqb \label{eqn-lfpp}
D_h^\ep(z,w) := \inf_{P : z\rta w} \int_0^1 e^{\frac{\gamma}{d_\gamma} h_\ep^*(P(t))} |P'(t)| \,dt ,\quad\forall z,w\in U
\eqe 
where the infimum is over all piecewise continuously differentiable paths from $z$ to $w$. 
It is shown in~\cite{dddf-lfpp} that there are deterministic positive scaling constants $\{\frk a_\ep\}_{\ep  >0}$ such that the laws of the random metrics $\frk a_\ep^{-1} D_h^\ep$ are tight w.r.t.\ the local uniform topology on $U\times U$, and moreover every possible subsequential limit of these laws is a metric which induces the Euclidean topology on $U$. It was subsequently shown in~\cite{gm-uniqueness}, building on~\cite{local-metrics,lqg-metric-estimates,gm-confluence}, that the subsequential limit is unique. In fact, the metrics $\frk a_\ep^{-1} D_h^\ep$ converge in probability to a limiting metric $D_h$ which is uniquely characterized by a list of axioms. We will review these axioms in Section~\ref{sec-lqg-metric} below.

In the special case when $\gamma=\sqrt{8/3}$, there is a completely different, earlier construction of the LQG metric due to Miller and Sheffield~\cite{lqg-tbm1,lqg-tbm2,lqg-tbm3} based on a process called \emph{quantum Loewner evolution}. It is shown in~\cite[Corollary 1.5]{gm-uniqueness} that the Miller-Sheffield $\sqrt{8/3}$-LQG metric is the same as the one obtained as the limit of~\eqref{eqn-lfpp} for $\gamma=\sqrt{8/3}$. Using their construction, Miller and Sheffield~\cite[Corollary 1.5]{lqg-tbm2} showed that certain special $\sqrt{8/3}$-LQG surfaces, viewed as metric spaces, are isometric to so-called \emph{Brownian surfaces}, such as the Brownian map. These Brownian surfaces are random metric spaces which describe the scaling limits of uniform random planar maps in the Gromov-Hausdorff topology, see, e.g.,~\cite{legall-uniqueness,miermont-brownian-map,bet-mier-disk,bmr-uihpq,gwynne-miller-uihpq}. 

Although the LQG metric induces the same topology as the Euclidean metric, many of its properties are very different from those of any smooth Riemannian distance function. 
For example, the LQG metric exhibits \emph{confluence of geodesics}~\cite{gm-confluence}. Roughly speaking, this means that two LQG geodesics with the same starting point and different target points typically coincide for a non-trivial interval of time; see Section~\ref{sec-confluence} for precise statements. 
As another example, the boundaries of LQG metric balls have Hausdorff dimension strictly larger than 1~\cite{gwynne-ball-bdy,lqg-zero-one} and have infinitely many connected components~\cite{lqg-zero-one}.

In this paper, we will further investigate the qualitative properties of LQG geodesics. In particular, we will show that a typical point is joined to any other point by at most three geodesics (Theorem~\ref{thm-zero-geo}); we will classify the possibly ``networks" of geodesics joining pairs of points which are dense in $\BB C\times \BB C$ (Theorem~\ref{thm-dense}); and we will prove that any two points in an LQG surface are joined by at most finitely many distinct geodesics (Theorem~\ref{thm-finite-geo}). 
Our classification of geodesic networks is the LQG analog of the classification of geodesic networks in the Brownian map from~\cite{akm-geodesics} (although the proof is quite different). 

A remarkable feature of the results in this paper is that several qualitative properties of LQG geodesics do not depend on $\gamma$. 
For example, the set of possible dense geodesic networks in Theorem~\ref{thm-dense} does not depend on $\gamma$. 
This is in contrast to \emph{quantitative} properties of LQG distances, such as the Hausdorff dimensions of various sets, which are expected to be $\gamma$-dependent. 
\medskip

\noindent\textbf{Acknowledgments.} We thank three anonymous referees for helpful comments on an earlier version of this article. We thank Jason Miller, Josh Pfeffer, Wei Qian, and Scott Sheffield for helpful discussions. The author was partially supported by a Clay research fellowship and a Trinity college, Cambridge junior research fellowship.

\subsection{Main results}
\label{sec-results}

For concreteness, throughout most of this paper we will restrict attention to the case when $U= \BB C$ and $h$ is the whole-plane Gaussian free field, normalized so that its average over the unit circle is zero (see~\cite[Section 2.2]{ig4} or~\cite[Section 5.4]{bp-lqg-notes} or~\cite[Section 3.2.2]{ghs-mating-survey} for background on the whole-plane GFF). 
Our results can be extended to variants of the GFF on other domains using local absolute continuity. 
We also fix $\gamma \in (0,2)$ and let $D_h$ be the $\gamma$-LQG metric associated with $h$, so that $D_h$ is a random metric on $\BB C$. 
In order for our main results to make sense, we need the following basic fact about the existence and uniqueness of LQG geodesics. 

\begin{lem} \label{lem-geo-basic}
Almost surely, for each distinct $z,w\in\BB C$ there is at least one $D_h$-geodesic from $z$ to $w$. 
For a fixed choice of $z$ and $w$, a.s.\ this $D_h$-geodesic is unique. 
\end{lem}
\begin{proof}
Almost surely, the metric space $(\BB C , D_h)$ is a boundedly compact length space: i.e., closed bounded sets are compact and the $D_h$-distance between any two points is the infimum of the $D_h$-lengths of continuous paths between them. See Axiom~\ref{item-metric-length} below and~\cite[Lemma 3.8]{lqg-metric-estimates}. 
Therefore, the existence of $D_h$-geodesics follows from general metric space theory~\cite[Corollary 2.5.20]{bbi-metric-geometry}. 
The uniqueness of the $D_h$-geodesic between fixed points is established in~\cite[Theorem 1.2]{mq-geodesics}. 
\end{proof}

Throughout this paper, for a set $A\subset\BB C$ we write $\dim_{\mcl H}^\gamma A$ for the Hausdorff dimension of the metric space $(A,D_h|_A)$ and we refer to $\dim_{\mcl H}^\gamma A$ as the \emph{$D_h$-Hausdorff dimension} of $A$. 
 
Our first main result concerns the number of distinct geodesics from a fixed point to an arbitrary point. 
Due to the translation invariance of the law of $h$, modulo additive constant, we can assume without loss of generality that the fixed point is the origin. 

\begin{thm} \label{thm-zero-geo}
Almost surely, for each $z\in\BB C$ there are either 1, 2, or 3 $D_h$-geodesics from 0 to $z$. 
Furthermore, a.s.\ 
\begin{enumerate}
\item For Lebesgue-a.e.\ $z\in\BB C$, there is a unique $D_h$-geodesic from 0 to $z$.
\item The set of points $z\in\BB C$ for which there are exactly two distinct $D_h$-geodesics from 0 to $z$ is dense in $\BB C$ and the $D_h$-Hausdorff dimension of this set is in $[1,d_\gamma-1]$. 
\item The set of points $z\in\BB C$ for which there are exactly three distinct $D_h$-geodesics from 0 to $z$ is dense in $\BB C$ and is countably infinite.
\end{enumerate}
\end{thm}

The analog of Theorem~\ref{thm-zero-geo} in the case of the Brownian map (equivalently, the case of a $\sqrt{8/3}$-LQG surface centered at a quantum typical point) follows from the fact that in the Brownian map, the dual of the geodesic tree is a continuum random tree. See~\cite{legall-geodesics} or the discussion at the beginning of~\cite[Section 1.3]{akm-geodesics}. In the case of $\gamma$-LQG for general $\gamma\in (0,2)$, however, we have no exact description of the laws of any functionals of the metric. Hence the proof of Theorem~\ref{thm-zero-geo} requires a non-trivial amount of work. In fact, the proof of Theorem~\ref{thm-zero-geo} will occupy most of Section~\ref{sec-geo-network}.  

We do not expect that either the lower bound or the upper bound for the Hausdorff dimension of points joined to zero by two distinct geodesics is optimal --- indeed, in the case of the Brownian map, equivalently $\sqrt{8/3}$-LQG, the dimension of the analogous set is 2 whereas $d_{\sqrt{8/3}} =4$. In our setting, the lower bound comes from the fact that the set of points joined to zero by multiple geodesics contains a non-trivial connected set (see Section~\ref{sec-three-geo-dense}) and the upper bound comes from a ``one-point estimate" argument (see Section~\ref{sec-multi-geo-dim}). 

\begin{remark} \label{remark-log-singularity}
Theorem~\ref{thm-zero-geo} is also true, with the same proof, if we replace $h$ by $h - \alpha \log |\cdot|$ for $\alpha \in (-\infty,Q)$. This is because the confluence of geodesics results for the LQG metric from~\cite{gm-confluence} and the uniqueness of LQG geodesics from~\cite[Theorem 1.2]{mq-geodesics} still hold at a point with an $\alpha$-log singularity, with the same proofs; see~\cite[Remark 1.5]{gm-confluence}. We require $\alpha < Q$ in order to ensure that the origin is at infinite $D_h$-distance from every other point, see~\cite[Theorem 1.11]{lqg-metric-estimates}.
In particular, by taking $\alpha=\gamma$ and using a standard property of the LQG area measure~\cite{kahane} (see also~\cite[Section 3.3]{shef-kpz}), we see that Theorem~\ref{thm-zero-geo} is true if we look at geodesics started from a typical point sampled from the LQG area measure $\mu_h$ instead of geodesics started from 0. 
\end{remark}

Using Theorem~\ref{thm-zero-geo} and the confluence of geodesics results for LQG surfaces from~\cite{gm-confluence,lqg-zero-one}, we can classify the possible topologies for networks of $D_h$-geodesics joining two points which are dense in $\BB C\times \BB C$. An identical classification in the case of the Brownian map is given in~\cite[Theorem 8]{akm-geodesics}.

\begin{defn} \label{def-normal-network}
For $(z,w) \in \BB C\times \BB C$ and $(n,m) \in \BB N$, we say that $(z,w)$ induces a \emph{normal $(n,m)$-geodesic network} if the following is true.
\begin{itemize}
\item There is a point $u \in \BB C$ such that every $D_h$-geodesic from $z$ to $w$ passes through $u$. 
\item There are exactly $n$ (resp.\ $m$) distinct $D_h$-geodesics from $u$ to $z$ (resp.\ $w$).
\item For any two $D_h$-geodesics $P,P'$ from $u$ to $z$, there exists $r\in (0,D_h(u,z))$ such that $P|_{[0,r]} = P'|_{[0,r]}$ and $P((r,D_h(u,z))) \cap P'((r,D_h(u,z))) = \emptyset$. Moreover, the same holds with $w$ in place of $z$. 
\end{itemize}
We write $N(n,m)$ for the set of pairs $(z,w) \in \BB C\times \BB C$ which induce a normal  $(n,m)$-geodesic network. 
\end{defn}

See Figure~\ref{fig-normal-network} for an illustration of Definition~\ref{def-normal-network}. We note that if $(z,w)$ induces a normal $(n,m)$-geodesic network, then there are exactly $nm$ distinct geodesics from $z$ to $w$.

\begin{figure}[t!]
 \begin{center}
\includegraphics[scale=.9]{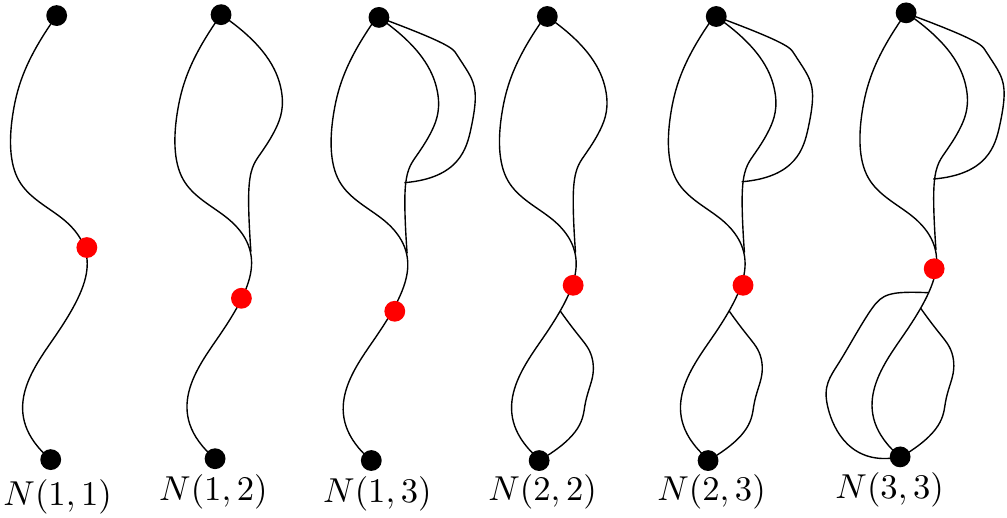}
\vspace{-0.01\textheight}
\caption{Illustration of normal $(n,m)$-geodesic networks for $n,m \in \{1,2,3\}$ (up to changing the order of the endpoints). In each network, a possible choice of the point $u$ of Definition~\ref{def-normal-network} is shown in red. By Theorem~\ref{thm-dense}, these are the only the types of geodesic networks whose pairs of endpoints are dense in $\BB C\times \BB C$. By Theorem~\ref{thm-zero-geo} and Lemma~\ref{lem-regular-geo'}, if the bottom endpoint of the geodesic is the origin then only the first three possibilities in the figure can arise.
}\label{fig-normal-network}
\end{center}
\vspace{-1em}
\end{figure}

\begin{thm} \label{thm-dense}
Almost surely, for any $n,m \in \{1,2,3\}$, $N(n,m)$ is dense in $\BB C \times \BB C$. Furthermore, a.s.\ $(\BB C\times \BB C) \setminus \bigcup_{n,m\in \{1,2,3\}} N(n,m)$ is nowhere dense in $\BB C\times \BB C$. 
\end{thm}

Theorem~\ref{thm-dense} does not give a complete classification of the possible geodesic networks for an LQG surface since there could be other configurations of geodesics besides $N(n,m)$ for $n,m\in\{1,2,3\}$ which occur for points $(z,w)$ in a nowhere dense subset of $\BB C\times \BB C$. In fact, such configurations are proven to exist in the case when $\gamma = \sqrt{8/3}$ in~\cite{mq-strong-confluence}. See Question~\ref{ques-network} and the discussion just after for more details.  

It is easy to see from the results of~\cite{gm-confluence,lqg-zero-one} that a.s.\ $(0,z) \in  N(1,m)$ for every $z\in\BB C$ such that there are $m$ distinct geodesics from 0 to $z$. 
For the sake of reference we record this fact in the following lemma.

\begin{lem} \label{lem-regular-geo'}
Almost surely, for each $z\in\BB C$ the following is true. 
\begin{enumerate}
\item For any $D_h$-geodesic $P$ from 0 to $z$ and any $t \in (0,D_h(0,z))$, $P|_{[0,t]}$ is the only $D_h$-geodesic from $0$ to $P(t)$. \label{item-regular-unique}
\item For any two distinct $D_h$-geodesics $P,P'$ from 0 to $z$, there exists $r = r(P,P') \in (0,D_h(0,z))$ such that $P|_{[0,r]} = P'|_{[0,r]}$ and $P((r,D_h(0,z))) \cap P'((r,D_h(0,z))) = \emptyset$. \label{item-regular-disjoint}
\item If there are exactly $m$ distinct geodesics from 0 to $z$ then $(0,z) \in N(1,m)$. \label{item-regular-normal}
\end{enumerate}  
\end{lem}

We note that once Theorem~\ref{thm-zero-geo} is established, Lemma~\ref{lem-regular-geo'} shows that  a.s.\ $(0,z) \in \bigcup_{m=1}^3 N(1,m)$ for every $z\in\BB C$. 

\begin{proof}[Proof of Lemma~\ref{lem-regular-geo'}]
Assertion~\ref{item-regular-unique} is just a re-statement of~\cite[Lemma 3.9]{lqg-zero-one}. 

To prove assertion~\ref{item-regular-disjoint}, let $P$ and $P'$ be distinct geodesics from 0 to $z$ and define 
\eqb
r = r(P,P') := \inf\left\{ t \in [0,D_h(0,z)]  : P(t) \not= P'(t) \right\} .
\eqe
Then $r < D_h(0,z)$ since $P$ and $P'$ are distinct. 
Furthermore, a.s.\ $r > 0$ for any two distinct geodesics started from 0 by confluence of geodesics~\cite[Theorem 1.3]{gm-confluence}. 
By definition, $P|_{[0,r]} = P'|_{[0,r]}$. 
We must show that $P((r,D_h(0,z))) \cap P'((r,D_h(0,z))) = \emptyset$.  
Indeed, if $u \in P((r,D_h(z,w))) \cap P'((r,D_h(z,w))) $ then $u = P(t) = P'(t)$ for $t = D_h(z,u) \in (r, D_h(z,w))$. 
Since $P|_{[0,t]} \not= P'|_{[0,t]}$ (by the definition of $r$), this implies that there are two distinct geodesics from 0 to $P(t)$, namely $P|_{[0,t]}$ and $P'|_{[0,t]}$. 
This cannot happen by condition~\ref{item-regular-unique}. 
Therefore, condition~\ref{item-regular-disjoint} holds. 

It remains to prove assertion~\ref{item-regular-normal}. 
To this end, let $P^1,\dots,P^m$ be the geodesics from 0 to $z$.
For distinct $i,j \in [1,m]_{\BB Z}$, let $r_{i,j}$ satisfy the condition of assertion~\ref{item-regular-disjoint} with $P = P^i$ and $P' = P^j$. 
Let $\ul r := \inf_{i,j} r_{i,j}$ and let $u = P^1(s)$ for some $s \in (0, \ul r)$. 
By inspection, the pair $(0,z)$ satisfies the conditions in Definition~\ref{def-normal-network} of $N(1,m)$ with this choice of $u$. 
\end{proof}

Our last main result gives a deterministic, finite upper bound for the maximal number of $D_h$-geodesics joining \emph{any} two points in $\BB C$.

\begin{thm} \label{thm-finite-geo}
There is a finite, deterministic number $m = m(\gamma )\in \BB N$ such that a.s.\ any two points in $\BB C$ are joined by at most $m$ distinct $D_h$-geodesics.
\end{thm}

By Theorem~\ref{thm-dense}, a.s.\ there exist points in $\BB C$ joined by 9 distinct geodesics, so $m\geq 9$. 
It is shown in~\cite[Theorem 1.6]{mq-strong-confluence} that for $\gamma=\sqrt{8/3}$, Theorem~\ref{thm-finite-geo} holds with $m=9$. We expect that the same is true for general $\gamma \in (0,2)$. 
However, the proof of~\cite[Theorem 1.6]{mq-strong-confluence} requires an \emph{a priori} finite upper bound on the number of geodesics joining any two points, see~\cite[Section 4.1]{mq-strong-confluence}. 
We expect that the same situation is true for $\gamma$-LQG for general $\gamma \in (0,2)$, i.e., Theorem~\ref{thm-finite-geo} is likely to be a necessary input in the proof that the maximal number of geodesics joining any two points is 9.  

Our proof of Theorem~\ref{thm-finite-geo} shows that $m(\gamma) $ is at most the largest integer for which 
\eqb \label{eqn-geo-count-bound}
\left\lfloor \frac12 \left( \left( \frac{17}{3} m  + 1 \right) - \sqrt{\left( \frac{17}{3} m  + 1 \right)^2 - 4m^2} \right) \right\rfloor  
\leq \lfloor 2d_\gamma+1 \rfloor  .
\eqe
For example, for $\gamma$ close to 0 we have $\lfloor 2d_\gamma+1 \rfloor = 5$ and so $m\leq 23$. 
Due to the bounds in~\cite{dg-lqg-dim}, for all $\gamma \in (0,2)$ we have $\lfloor 2d_\gamma+1 \rfloor \leq 10$ and so $m\leq 50$. 
The strange looking quantity on the left side of~\eqref{eqn-geo-count-bound} comes from a lower bound for the size of a so-called \emph{independent set} in a general graph from~\cite{hs-independence-number}, see Lemma~\ref{lem-independent-set}. It is possible that our bound for $m$ could be improved by soft arguments. However, we do not try to optimize it since, as discussed above, we expect that in actuality $m=9$, and we do not believe that this value of $m$ cannot be obtained by purely soft arguments.

\begin{remark}
The parameter $\gamma/d_\gamma$ in~\eqref{eqn-lfpp} lies in $(0,2/d_2)$ where $d_2 = \lim_{\gamma \rta 2^-} d_\gamma$. The estimates for $d_\gamma$ from~\cite{dg-lqg-dim,gp-lfpp-bounds} (see~\cite[Theorem 2.3]{gp-lfpp-bounds}) show that $ 0.4135 \leq 2/d_2 \leq 0.4189$.
The recent paper~\cite{dg-supercritical-lfpp} showed that if we define LFPP as in~\eqref{eqn-lfpp} with $\gamma/d_\gamma$ replaced by a parameter $\xi > 2/d_2$, then the LFPP metrics, re-scaled appropriately, admit subsequential limiting w.r.t.\ a certain topology. These subsequential limits are metrics on $\BB C$ which do not induce the Euclidean topology: rather, there is an uncountable dense set of points which lie at infinite distance from every other point. 
These metrics are expected to be related to LQG with ``matter central charge" in $(1,25)$ (for $\gamma \in (0,2)$, the matter central charge is $25-6(2/\gamma+\gamma/2)^2 \in (-\infty,1)$). 

It would be of interest to determine to what extent the results of this paper remain true in the case when $\xi  >2/d_2$. 
The confluence of geodesics results from~\cite{gm-confluence} can be extended to the case when $\xi > 2/d_2$; see~\cite{dg-confluence}. However, the arguments of the present paper do not apply in this setting since we use that the $\gamma$-LQG metric for $\gamma \in (0,2)$ induces the Euclidean topology (used frequently throughout our proofs) and the fact that each compact subset of $\BB C$ has finite upper Minkowski dimension w.r.t.\ the $\gamma$-LQG metric (used in the proof of Theorem~\ref{thm-finite-geo}); see Lemma~\ref{lem-ball-cover}.  
\end{remark}

\subsection{Open problems}
\label{sec-open-problems}

Recall that Theorem~\ref{thm-dense} only classifies geodesic networks which occur for a dense set of pairs of points in $\BB C\times\BB C$.

\begin{ques} \label{ques-network}
Give a complete classification of the possible topological configurations of geodesics joining pairs of points in $\BB C$, including configurations which occur for a nowhere dense set of points in $\BB C\times\BB C$.  
\end{ques}

Miller and Qian~\cite{mq-strong-confluence} give a nearly complete answer to Question~\ref{ques-network} in the $\gamma = \sqrt{8/3}$ (Brownian surface) case. 
Indeed,~\cite[Theorem 1.5]{mq-strong-confluence} gives an explicit finite list of topological configurations which can arise for the set of geodesics joining two arbitrary points in the Brownian map, which includes the normal $(n,m)$-networks for $n,m\in\{1,2,3\}$ as well as configurations which are not normal $(n,m)$-networks for any $n,m \in\BB N$. The theorem also gives upper bounds for the Hausdorff dimensions of the set of pairs of points $(z,w)$ in the Brownian map for which each of these configurations arise. 
Moreover,~\cite[Theorem 1.6]{mq-strong-confluence} shows that there exist pairs of points in the Brownian map joined by exactly $k$ distinct geodesics if and only if $k\in \{1,\dots,9\}$ and gives the Hausdorff dimension of the set of such pairs in terms of $k$. The lower bound for the dimension of the set of pairs of point joined by $k$ geodesics for $k\in \{1,2,3,4,6,9\}$ was already established in~\cite{akm-geodesics}. Note that for $k\in \{5,7,8\}$, the pairs of points joined by exactly $k$ geodesics do not belong to $N(n,m)$ for any $n,m$. 
We find it likely that all of the configurations from~\cite[Theorem 1.5]{mq-strong-confluence} actually arise in the Brownian map, but Miller and Qian do not prove this in all cases, so they do not quite give a complete answer to Question~\ref{ques-network} for $\gamma=\sqrt{8/3}$. 

In light of the similarity between Theorem~\ref{thm-dense} and~\cite[Theorem 8]{akm-geodesics}, a natural guess is that the results of~\cite{mq-strong-confluence} remain true (with the same list of possible geodesic networks) for general values of $\gamma \in (0,2)$. However,~\cite{mq-strong-confluence} relies on exact independence properties for the Brownian map which are not expected to be true for general $\gamma \in (0,2)$, so novel ideas would be needed to extend their results.

\begin{ques} \label{ques-dim}
For each possible configuration of geodesics in Question~\ref{ques-network}, compute the Hausdorff dimension of the set of pairs of points $(z,w) \in \BB C\times \BB C$ joined by this type of configuration, w.r.t.\ the Euclidean and $\gamma$-LQG metrics.
\end{ques}

As discussed above,~\cite{mq-strong-confluence} gives a nearly complete answer to Question~\ref{ques-dim} in the case of LQG dimensions for $\gamma = \sqrt{8/3}$. In the general case, we do not even have a conjecture for most of the dimensions involved. 

It appears that the biggest obstacle to resolving Question~\ref{ques-dim} is to compute the Euclidean and $\gamma$-LQG dimensions of the set of points joined to 0 by exactly 2 distinct geodesics. Indeed, if we denote these dimensions by $\Delta_0$ and $\Delta_\gamma$, respectively, then we expect that 
\allb \label{eqn-dim-conj-0}
&\dim^0 N(1,1) = 4, \quad \dim^0 N(1,2) = 2 + \Delta_0, \quad \dim^0 N(1,3) = 2 ,\notag\\ 
&\dim^0 N(2,2) = 2\Delta_0, \quad \dim^0 N(2,3) = \Delta_0 , \quad \dim^0 N(3,3) = 0
\alle
and
\allb \label{eqn-dim-conj-gamma}
&\dim^\gamma N(1,1) = 2d_\gamma , \quad \dim^\gamma N(1,2) = d_\gamma + \Delta_\gamma , \quad \dim^\gamma N(1,3) = d_\gamma , \notag\\
&\dim^\gamma N(2,2) = 2\Delta_\gamma, \quad \dim^\gamma N(2,3) = \Delta_\gamma , \quad \dim^\gamma N(3,3) = 0 ,
\alle
where $\dim^0$ and $\dim^\gamma$ denote the Euclidean and $\gamma$-LQG Hausdorff dimensions, respectively. We also expect that it is not too hard to derive predictions for the Hausdorff dimension of the pairs of points joined by the other geodesic networks appearing in~\cite[Theorem 1.5]{mq-strong-confluence} in terms of $d_\gamma,\Delta_0,\Delta_\gamma$. 

Note that~\eqref{eqn-dim-conj-gamma} is consistent with the results of~\cite{akm-geodesics,mq-strong-confluence} since (as explained in~\cite{akm-geodesics}) we have $\Delta_{\sqrt{8/3}} = 2$ for $\gamma = \sqrt{8/3}$. 
We do not have a conjecture for the value of $\Delta_0$ for any value of $\gamma$ or for $\Delta_\gamma$ for $\gamma\not=\sqrt{8/3}$.

\section{Preliminaries}
\label{sec-prelim}

In this section, we fix some more or less standard notation (Section~\ref{sec-basic}), then review some known properties of the LQG metric: the axiomatic characterization (Section~\ref{sec-lqg-metric}), the definition and properties of filled metric balls (Section~\ref{sec-filled}), and the confluence of geodesics property (Section~\ref{sec-confluence}). 

\subsection{Basic notation}
\label{sec-basic}

\noindent
We write $\BB N = \{1,2,3,\dots\}$ and $\BB N_0 = \BB N \cup \{0\}$. 
For $a < b$, we define the discrete interval $[a,b]_{\BB Z}:= [a,b]\cap\BB Z$. 
\medskip
 
\noindent
If $f  :(0,\infty) \rta \BB R$ and $g : (0,\infty) \rta (0,\infty)$, we say that $f(\ep) = O_\ep(g(\ep))$ (resp.\ $f(\ep) = o_\ep(g(\ep))$) as $\ep\rta 0$ if $f(\ep)/g(\ep)$ remains bounded (resp.\ tends to zero) as $\ep\rta 0$. 
We similarly define $O(\cdot)$ and $o(\cdot)$ errors as a parameter goes to infinity.  
We often specify requirements on the dependencies on rates of convergence in $O(\cdot)$ and $o(\cdot)$ errors in the statements of lemmas/propositions/theorems, in which case we implicitly require that errors, implicit constants, etc., in the proof satisfy the same dependencies. 
\medskip

\noindent
For $z\in\BB C$ and $r>0$, we write $B_r(z)$ for the Euclidean ball of radius $r$ centered at $z$.  
\medskip 

\noindent
For a metric space $(X,D)$, $A\subset X$, and $r>0$, we write $\mcl B_r(A;D)$ for the open ball consisting of the points $x\in X$ with $D (x,A) < r$.  
If $A = \{y\}$ is a singleton, we write $\mcl B_r(\{y\};D) = \mcl B_r(y;D)$.

\subsection{Axiomatic characterization of the LQG metric}
\label{sec-lqg-metric}

In this section we review the axiomatic characterization of the LQG metric which was proven in~\cite{gm-uniqueness}. 
We will not need the uniqueness part of this characterization theorem for our proofs, but we will frequently use the fact that the LQG metric satisfies the axioms in the characterization theorem (which was checked in~\cite{lqg-metric-estimates,gm-uniqueness}). 
To state the axioms we need some preliminary definitions. 

\begin{defn} \label{def-metric-stuff}
Let $(X,D)$ be a metric space.
\begin{itemize}
\item
For a curve $P : [a,b] \rta X$, the \emph{$D$-length} of $P$ is defined by 
\eqbn
\op{len}\left( P ; D  \right) := \sup_{T} \sum_{i=1}^{\# T} D(P(t_i) , P(t_{i-1})) 
\eqen
where the supremum is over all partitions $T : a= t_0 < \dots < t_{\# T} = b$ of $[a,b]$. Note that the $D$-length of a curve may be infinite.
\item
We say that $(X,D)$ is a \emph{length space} if for each $x,y\in X$ and each $\ep > 0$, there exists a curve of $D$-length at most $D(x,y) + \ep$ from $x$ to $y$. 
\item
For $Y\subset X$, the \emph{internal metric of $D$ on $Y$} is defined by
\eqb \label{eqn-internal-def}
D(x,y ; Y)  := \inf_{P \subset Y} \op{len}\left(P ; D \right) ,\quad \forall x,y\in Y 
\eqe 
where the infimum is over all paths $P$ in $Y$ from $x$ to $y$. 
Note that $D(\cdot,\cdot ; Y)$ is a metric on $Y$, except that it is allowed to take infinite values.  
\item
If $X$ is an open subset of $\BB C$, we say that $D$ is  a \emph{continuous metric} if it induces the Euclidean topology on $X$. 
We equip the set of continuous metrics on $X$ with the local uniform topology on $X\times X$ and the associated Borel $\sigma$-algebra.
\end{itemize}
\end{defn}

We now give the axiomatic definition of the LQG metric. Since we will only be working with the whole-plane GFF, we only give the definition in the whole-plane case. 
See~\cite{gm-coord-change} for the axioms in the case of a general open domain $U\subset\BB C$. The definition involves the parameters 
\eqb \label{eqn-xi-Q}
\xi := \frac{\gamma}{d_\gamma} \quad \text{and} \quad Q := \frac{2}{\gamma}  + \frac{\gamma}{2} ,
\eqe
which appear in the transformation rules for the metric under adding a continuous function and applying a conformal map, respectively. Recall from the discussion just before~\eqref{eqn-lfpp} that $d_\gamma$ is the Hausdorff dimension of the LQG metric. 

\begin{defn}[The LQG metric]
\label{def-lqg-metric}
Let $\mcl D'(\BB C)$ be the space of distributions (generalized functions) on $\BB C$, equipped with the usual weak topology.   
A \emph{$\gamma$-LQG metric} is a measurable function $h\mapsto D_h$ from $\mcl D'(\BB C)$ to the space of continuous metrics on $\BB C$ with the following properties. 
Let $h$ be a \emph{whole-plane GFF plus a continuous function}, i.e., $h$ is a random distribution on $\BB C$ which can be coupled with a random continuous function $f$ in such a way that $h-f$ has the law of the whole-plane GFF.  Then the associated metric $D_h$ satisfies the following axioms.
\begin{enumerate}[I.]
\item \textbf{Length space.} Almost surely, $(\BB C,D_h)$ is a length space, i.e., the $D_h$-distance between any two points of $\BB C$ is the infimum of the $D_h$-lengths of $D_h$-continuous paths (equivalently, Euclidean continuous paths) between the two points. \label{item-metric-length}
\item \textbf{Locality.} Let $V \subset \BB C$ be a deterministic open set. 
The $D_h$-internal metric $D_h(\cdot,\cdot ; V)$ is a.s.\ equal to $D_{h|_V}$, so in particular it is a.s.\ determined by $h|_V$.  \label{item-metric-local}
\item \textbf{Weyl scaling.} Let $\xi = \gamma/d_\gamma$ be as in~\eqref{eqn-xi-Q}. For a continuous function $f : \BB C \rta \BB R$, define
\eqb \label{eqn-metric-f}
(e^{\xi f} \cdot D_h) (z,w) := \inf_{P : z\rta w} \int_0^{\op{len}(P ; D_h)} e^{\xi f(P(t))} \,dt , \quad \forall z,w\in \BB C,
\eqe 
where the infimum is over all continuous paths from $z$ to $w$ parametrized by $D_h$-length.
Then a.s.\ $ e^{\xi f} \cdot D_h = D_{h+f}$ for every continuous function $f: \BB C \rta \BB R$. \label{item-metric-f}
\item \textbf{Conformal coordinate change.} Let $a \in \BB C\setminus \{0\}$ and $b\in\BB C$. 
Then, with $Q = 2/\gamma+\gamma/2$ as in~\eqref{eqn-xi-Q}, a.s.\ \label{item-metric-coord}
\eqb \label{eqn-metric-coord}
 D_h \left( a z + b ,a w + b\right) = D_{h(a\cdot+b) + Q\log |a|}\left( z , w \right)  ,\quad  \forall z,w \in \BB C .
\eqe    
\end{enumerate}
\end{defn}

It is shown in~\cite{gm-uniqueness}, building on~\cite{dddf-lfpp,local-metrics,lqg-metric-estimates,gm-confluence} that the limit of the metrics $D_h^\ep$ of~\eqref{eqn-lfpp} satisfies the axioms of Definition~\ref{def-lqg-metric}. 
Furthermore, the metric satisfying these axioms are unique in the following sense.
If $D$ and $\wt D$ are two such metrics, then there is a deterministic constant $C>0$ such that whenever $h$ is a whole-plane GFF plus a continuous function, a.s.\ $\wt D_h = C D_h$.

\subsection{Filled metric balls}
\label{sec-filled}
 
For $z, w\in\BB C$ and $s \in [0,D_h(w,z))$, we define the \emph{filled metric ball centered at $w$ and targeted at $z$}, denoted $\mcl B^{z,\bullet}_s(w;D_h)$, to be the union of the closed metric ball $\ol{\mcl B_s(w;D_h)}$ and the set of points which this closed metric ball disconnects from $z$. 
For $s \geq D_h(0,z)$ we set $\mcl B^{z,\bullet}_s(w;D_h) := \BB C$. 
Note that $\mcl B^{z,\bullet}(w;D_h)$ is unbounded if $\ol{\mcl B_s(w;D_h)}$ disconnects $z$ from $\infty$.
For $s> 0$, we similarly define the \emph{filled metric ball centered at $w$ and targeted at $\infty$} to be the set $\mcl B^\bullet_s(w;D_h) = \mcl B^{\infty,\bullet}_s(w;D_h)$ 
which is the union of $\ol{\mcl B_s(w;D_h)}$ and the set of points which it disconnects from $\infty$. 
We will most often work with filled metric balls centered at zero, so to lighten notation we abbreviate
\eqb
\mcl B_s^{z,\bullet} := \mcl B_s^{z,\bullet}(0;D_h) \quad \text{and} \quad \mcl B_s^\bullet := \mcl B_s^\bullet(0;D_h) .
\eqe

The following lemma is proven in~\cite[Section 2.1]{lqg-zero-one} (see also~\cite[Proposition 2.1]{tbm-characterization}).

\begin{lem}  \label{lem-metric-curve}
Almost surely, for each $z,w\in \BB C$ and each $s \in (0,D_h(w,z))$, the set $\bdy\mcl B_s^{z,\bullet}(w;D_h)$ is a Jordan curve.
\end{lem}

Lemma~\ref{lem-metric-curve} together with Carath\'eodory's theorem~\cite[Theorem 2.6]{pom-book} implies that if $\BB C\setminus \mcl B_s^{z,\bullet}(w;D_h)$ is bounded, then every conformal map $\BB D \rta \BB C\setminus \mcl B_s^{z,\bullet}(w;D_h)$ extends to a homeomorphism $\ol{\BB D} \rta \ol{\BB C\setminus \mcl B_s^{z,\bullet}(w;D_h)}$. If $\BB C\setminus \mcl B_s^{z,\bullet}(w;D_h)$ is unbounded, we get a similar result if we replace $\BB C\setminus \mcl B_s^{z,\bullet}(w;D_h)$ by $(\BB C\setminus \mcl B_s^{z,\bullet}(w;D_h)) \cup \{\infty\}$, viewed as a subset of the Riemann sphere.

\subsection{Confluence of geodesics}
\label{sec-confluence}

Here we review some results on confluence of geodesics in LQG surfaces which were proven in~\cite{gm-confluence,lqg-zero-one}. 
Roughly speaking, these results say that (unlike geodesics in a smooth Riemannian manifold) LQG geodesics tend to ``merge into one another" and stay together for non-trivial intervals of time. 
These confluence results will be the key tool in the proofs of our main theorems. 
The simplest version of confluence is the following result, which is~\cite[Theorem 1.2]{gm-confluence}.

\begin{thm}[Confluence at a typical point] \label{thm-clsce}
Fix $z\in\BB C$. Almost surely, for each radius $s > 0$ there exists a radius $t \in (0,s)$ such that any two $D_h$-geodesics from $z$ to points outside of $\mcl B_s(z;D_h)$ coincide on the time interval $[0,t]$.
\end{thm} 

Theorem~\ref{thm-clsce} gives an a.s.\ confluence property for geodesics started from a \emph{fixed} point of $\BB C$, but it does not tell us anything about geodesics with arbitrary starting and target points. The following improvement on Theorem~\ref{thm-clsce} was proven in~\cite[Theorem 1.2]{lqg-zero-one}. Roughly speaking, it says that the result of Theorem~\ref{thm-clsce} extends to geodesics whose starting points are ``nearby" a typical point.

\begin{thm}[Confluence near a typical point] \label{thm-general-confluence}
Fix $z\in\BB C$. Almost surely, for each  neighborhood $U$ of $z$ there is a neighborhood $U'\subset U$ of $z$ and a point $Z \in U\setminus U'$ such that every $D_h$-geodesic from a point in $U'$ to a point in $\BB C\setminus U$ passes through $Z$.
\end{thm}

Recall that for $z\in\BB C\setminus \{0\}$ and $s\in [0,D_h(0,z)]$, $\mcl B_s^{z,\bullet} = \mcl B_s^{z,\bullet}(0;D_h)$ is the filled metric ball centered at 0 and targeted at $z$.  
Each point $x\in \bdy \mcl B_s^{z,\bullet}  $ lies at $D_h$-distance exactly $s$ from $0$, so every $D_h$-geodesic from $0$ to $x$ stays in $\mcl B_s^{z,\bullet} $. For some points $x$ there might be many such $D_h$-geodesics. But, it is shown in~\cite[Lemma 2.4]{gm-confluence} that there is always a distinguished $D_h$-geodesic from 0 to $x$, called the \emph{leftmost geodesic}, which lies (weakly) to the left of every other $D_h$-geodesic from 0 to $x$ if we stand at $x$ and look outward from $\mcl B_s^\bullet$. Strictly speaking,~\cite[Lemma 2.4]{gm-confluence} only treats the case of filled metric balls targeted at $\infty$, but the same proof works for filled metric balls with different target points. 
The following proposition is~\cite[Proposition 3.6]{lqg-zero-one}, and is a generalization of~\cite[Theorem 1.3]{gm-confluence} (which is the version for filled metric balls target at $\infty$). 

\begin{prop}[Confluence across a filled metric annulus] \label{prop-conf-finite} 
Almost surely, for each $w\in\BB C\setminus \{0\}$ and each $0 < t < s < D_h(0,w)$, the following is true.
\begin{enumerate}
\item There is a finite set of points $\mcl X = \mcl X_{t,s}^w \subset \bdy\mcl B_t^{w,\bullet}$ such that every leftmost $D_h$-geodesic from 0 to a point of $\bdy \mcl B_s^{w,\bullet}$ passes through some $x\in\mcl X$. \label{item-conf-finite}
\item There is a unique $D_h$-geodesic from 0 to $x$ for each $x\in \mcl X$. \label{item-conf-finite-unique}
\item For $x\in \mcl X$, let $I_x$ be the set of $y\in\bdy\mcl B_s^{w,\bullet}$ such that that the leftmost $D_h$-geodesic from 0 to $y$ passes through $x$.
Each $I_x $ for $x\in\mcl X $ is a connected arc of $\bdy\mcl B_s^{w,\bullet} $ (possibly a singleton) and $\bdy\mcl B_s^{w,\bullet}$ is the disjoint union of the arcs $I_x $ for $x\in \mcl X $. \label{item-conf-finite-arcs}
\item The counterclockwise cyclic ordering of the arcs $I_x $ is the same as the counterclockwise cyclic ordering of the corresponding points $x\in \mcl X \subset\bdy \mcl B_t^{w,\bullet}$. \label{item-conf-finite-order} 
\end{enumerate}
\end{prop}

Proposition~\ref{prop-conf-finite} concerns only \emph{leftmost} $D_h$-geodesics. We will also sometimes need to work with geodesics which are not necessarily leftmost. The following proposition, which is a re-statement of~\cite[Proposition 3.7]{lqg-zero-one}, will allow us to do so. 

\begin{prop} \label{prop-conf-endpt-finite}
Fix $0<t < s$ and let $\mcl X = \mcl X_{t,s}^w$ be the set of confluence points as in Proposition~\ref{prop-conf-finite}. 
Almost surely, on the event $\{s < D_h(0,w)\}$, the following is true. For every $D_h$-geodesic $P$ from 0 to a point of $\BB C\setminus \mcl B_s^{w,\bullet}$ there is an $x\in\mcl X$ such that $P(t)  = x$ and $P(s) $ is a point of the arc $I_x$ which is not one of the endpoints of $I_x$. 
\end{prop} 

 Note that, unlike Proposition~\ref{prop-conf-finite}, Proposition~\ref{prop-conf-endpt-finite} holds only for a \emph{fixed} choice of $t$ and $s$, not for all choices of $t$ and $s$ simultaneously.

\section{Classification of geodesic networks}
\label{sec-geo-network}

The purpose of this section is to prove the parts of our main theorems which do not require quantitative estimates. In particular, we will prove all of the assertions of Theorem~\ref{thm-zero-geo} except for the upper bound for the Hausdorff dimension of the set of points joined to zero by two distinct geodesics; and we will prove Theorem~\ref{thm-dense}. 
We start in Section~\ref{sec-three-geo-countable} by showing that there are at most countably many points $z\in\BB C$ for which there are at least three distinct $D_h$-geodesics from 0 to $z$. This is done using confluence of geodesics together with a purely topological argument. 
In Section~\ref{sec-four-geo}, we show that there cannot be any points in $\BB C$ which are joined to zero by four distinct geodesics by combining the countability result of Section~\ref{sec-three-geo-countable} with a perturbation argument based on the Weyl scaling property of the LQG metric and the absolute continuity properties of the GFF. 

In Section~\ref{sec-three-geo-dense}, we prove that there are \emph{at least} countably many points joined to zero by three distinct $D_h$-geodesics and \emph{at least} a Hausdorff-dimension 1 set of points joined to zero by two distinct $D_h$-geodesics. The proof is again based on confluence together with `soft" topological arguments.
In Section~\ref{sec-zero-geo} we collect the results of the preceding subsections to prove most of Theorem~\ref{thm-zero-geo}. 
In Section~\ref{sec-normal-dense} we deduce Theorem~\ref{thm-dense} from confluence, Theorem~\ref{thm-zero-geo}, and a short argument similar to ones in~\cite{akm-geodesics} (this is the only part of the paper which is close to the proofs in~\cite{akm-geodesics}).

\subsection{At most countably many points joined to the origin by at least three distinct geodesics}
\label{sec-three-geo-countable}

\begin{prop} \label{prop-three-geo-countable}
Almost surely, there are at most countably many points $z\in\BB C$ for which there are three or more distinct $D_h$-geodesics from 0 to $z$.
\end{prop}

The proof of Proposition~\ref{prop-three-geo-countable} proceeds by way of a simple topological argument which is illustrated in Figure~\ref{fig-three-geo}.  
The following lemma is the key input in the proof. 

\begin{figure}[t!]
 \begin{center}
\includegraphics[scale=.9]{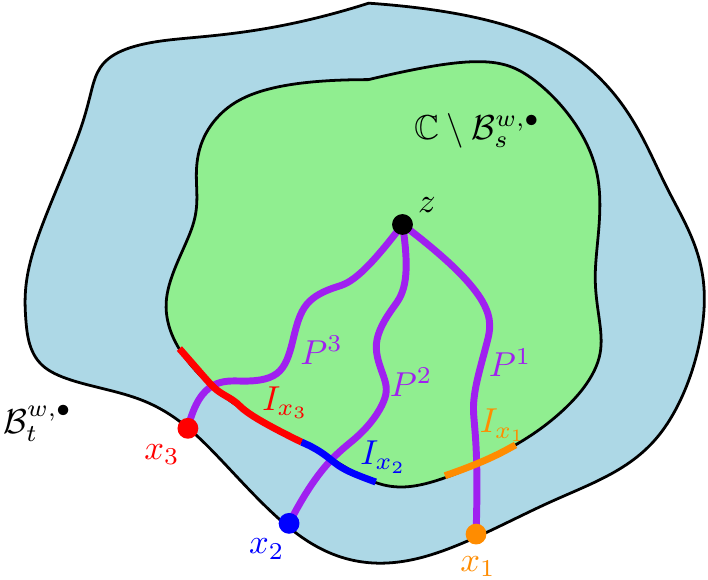}
\vspace{-0.01\textheight}
\caption{Illustration of the proof of Lemma~\ref{lem-three-geo-pts}. The three $D_h$-geodesics $P^1,P^2,P^3$ from 0 to $z$ divide $\BB C\setminus \mcl B_s^{w,\bullet}$ into three connected components. For each of these connected components, one of the arcs $I_{x_1} , I_{x_2}$, or $I_{x_3}$ does not intersect the boundary of the component. Consequently, if there is another point $\wt z \not=z$ and three $D_h$-geodesics $\wt P^1,\wt P^2,\wt P^3$ from 0 to $\wt z$ which pass through $x_1,x_2,x_3$, respectively, then at least one of the $\wt P^j$'s must cross one of the $P^i$'s. This is impossible since two $D_h$-geodesics started from 0 cannot cross without merging into each other. 
}\label{fig-three-geo}
\end{center}
\vspace{-1em}
\end{figure}

\begin{lem} \label{lem-three-geo-pts}
For each fixed $w\in\BB C$ and $0 < t < s$, the following is true a.s.\ on the event $\{s < D_h(0,w) \}$. 
Let $\mcl X = \mcl X_{t,s}^w \subset\bdy\mcl B_t^{w,\bullet}$ be the set of confluence points as in Proposition~\ref{prop-conf-finite} and let $x_1,x_2,x_3 \in \mcl X $ be distinct. 
There is at most one point $z \in \BB C\setminus \mcl B_s^{w,\bullet}$ with the following property: there are three $D_h$-geodesics $P^1,P^2,P^3$ from 0 to $z$ which pass through $x_1,x_2,x_3$, respectively.   
\end{lem}
\begin{proof}
See Figure~\ref{fig-three-geo} for an illustration. Basically, the idea of the proof is that if there were two points $z,\wt z$ with the property in the lemma statement, then topological considerations would force a geodesic from 0 to $z$ and a geodesic from 0 to $\wt z$ to cross. This cannot happen due to Lemma~\ref{lem-regular-geo'}. 

Throughout the proof we assume that $s < D_h(0,w)$ and there is a point $z\in\BB C\setminus  \mcl B_s^{w,\bullet}$ and $D_h$-geodesics $P^1,P^2,P^3$ from 0 to $z$ which pass through $x_1,x_2,x_3$, respectively (if such a point does not exist the lemma statement is vacuous). 
Define the arcs $I_{x_i} \subset \bdy\mcl B_s^{w,\bullet}$ for $i\in \{1,2,3\}$ as in Proposition~\ref{prop-conf-finite}.
Note that these arcs are disjoint. 
On the full-probability event of Proposition~\ref{prop-conf-endpt-finite}, for each $i\in \{1,2,3\}$, we have $P^i(s) \in I_{x_i}$ and $P^i(s)$ is not one of the endpoints of $I_{x_i}$.  

By Lemma~\ref{lem-regular-geo'}, a.s.\ $ P^i([s,t]) \cap P^j([s,t]) = \{z\}$ for each distinct $i,j\in \{1,2,3\}$: indeed, each point in this intersection is joined to zero by at least two distinct geodesics. 
Since $\BB C\setminus \mcl B_s^{w,\bullet}$ is conformally equivalent to either the disk or the plane minus a disk, we infer that the set 
\eqb \label{eqn-three-geo-divide}
\BB C\setminus \left( \mcl B_s^{w,\bullet} \cup \bigcup_{i=1}^3 P^i([s,D_h(0,z)])  \right)
\eqe
has exactly three connected components.  
Since none of $P^1,P^2,P^3$ hits an endpoint of one of the arcs $I_{x_i}$ for $i\in \{1,2,3\}$, the boundary of each of the above three connected components intersects exactly two of the arcs $I_{x_i} $ for $i\in \{1,2,3\}$. 

Now suppose $\wt z \in \BB C\setminus \mcl B_s^{w,\bullet}$, $\wt z \not=z$. 
We will show that there cannot be three $D_h$-geodesics from 0 to $\wt z$ which pass through each of $x_1$, $x_2$, and $x_3$.
If $\wt z \in \bigcup_{i=1}^3 P^i([s,D_h(0,z)]) $, then by Lemma~\ref{lem-regular-geo'}, there is a unique $D_h$-geodesic from 0 to $\wt z$. 
Hence we can assume that $\wt z$ lies in one of the connected components of the set~\eqref{eqn-three-geo-divide}.
Call this connected component $U$.
By the preceding paragraph, $\bdy U$ intersects only two of the arcs $I_{x_i}(s)$ for $i\in \{1,2,3\}$; assume without loss of generality that $\bdy U$ intersects $I_{x_1} $ and $I_{x_2} $. 

No $D_h$-geodesic started from 0 can hit $\bdy\mcl B_s^{w,\bullet}$ more than once.
So, in order for a $D_h$-geodesic $\wt P$ from $0$ to $\wt z$ to hit $I_{x_3} $ it must cross $P^i([s,t])$ for some $i\in \{1,2,3\}$.
In other words, there must be a time $u \in [s,t]$, a time $v \in [0,D_h(0,\wt z)]$, and an $i\in \{1,2,3\}$ such that $P^i(u) = \wt P(v)$ and $P^i|_{[0,u]} \not= \wt P|_{[0,v]}$. 
We will now argue that this cannot be the case.
There are two distinct $D_h$-geodesics from 0 to $P^i(u)$, namely $P^i|_{[0,u]}$ and $\wt P|_{[0,v]}$. 
By Lemma~\ref{lem-regular-geo'}, this means that $P^i(u) = z$ (so in particular $u = v = D_h(0,z)$). 
But then $z$ is hit by the $D_h$-geodesic $\wt P$ started from 0 at a time which is not the terminal time of $\wt P$.
By Lemma~\ref{lem-regular-geo'}, this implies that there is a unique $D_h$-geodesic from 0 to $z$, which is contrary to our choice of $z$. 
Therefore, no $D_h$-geodesic from 0 to $\wt z$ can hit $I_{x_3} $. By Proposition~\ref{prop-conf-endpt-finite}, this means that no $D_h$-geodesic from 0 to $\wt z$ can hit $x_3$. 
\end{proof}

We will also need the following trivial consequence of Lemma~\ref{lem-regular-geo'}.

\begin{lem} \label{lem-geo-distinct}
Almost surely, the following is true.
Let $z\in\BB C$, let $n\in\BB N$, and let $P^1,\dots,P^n$ be distinct $D_h$-geodesics from 0 to $z$. 
There exists $r_* \in (0,D_h(0,z))$ such that the geodesic segments $P^i([r_*,D_h(0,z))) $ for $i\in \{1,\dots,n\}$ are pairwise disjoint.
\end{lem}
\begin{proof}
Throughout the proof, all statements are required to hold a.s.\ for every possible choice of $z$ and $P^1,\dots,P^n$. 
By assertion~\ref{item-regular-disjoint} of Lemma~\ref{lem-regular-geo'}, for each distinct $i,j \in \{1,\dots,n\}$ there exists $r_{i,j} \in (0,D_h(0,z))$ such that $P^i([r_{i,j} , D_h(0,z)) \cap P^j([r_{i,j} , D_h(0,z))) =\emptyset$.  
Then the lemma statement holds with
\eqbn
r_* := \max\left\{ r_{i,j} : i,j \in \{1,\dots,n\}, i\not=j\right\} .
\eqen
\end{proof}

\begin{proof}[Proof of Proposition~\ref{prop-three-geo-countable}]
For $w\in\BB Q^2$ and $t,s\in\BB Q \cap (0,\infty)$ with $t < s$, let $ \mcl Z_{t,s}^w$ be the set of $z\in\BB C \setminus \mcl B_s^{w,\bullet}$ such that there are three $D_h$-geodesics $P^1,P^2,P^3$ from 0 to $z$ such that $P^1(t), P^2(t)$, and $P^3(t)$ are distinct. 

We claim that a.s.\ $\mcl Z_{t,s}^w$ is a finite set. 
To see this, define $\mcl X = \mcl X_{t,s}^w$ as in Proposition~\ref{prop-conf-finite}.
By Proposition~\ref{prop-conf-endpt-finite}, a.s.\ each $D_h$-geodesic from 0 to a point of $\BB C\setminus \mcl B_s^{w,\bullet}$ passes through some $x\in\mcl X$, necessarily at time $t$ since $\mcl X\subset \bdy \mcl B_t^{w,\bullet}$. 
By Lemma~\ref{lem-three-geo-pts}, for each distinct $x_1,x_2,x_3 \in \mcl X$, a.s.\ there is at most one $z\in \mcl Z_{t,s}^w  $ whose corresponding $D_h$-geodesics $P^1,P^2,P^3$ pass through $x_1,x_2,x_3$, respectively (necessarily at time $t$). Since the points $P^1(t) ,P^2(t),P^3(t)$ are assumed to be distinct, it follows that
\eqb
\#  \mcl Z_{t,s}^w    \leq (\#\mcl X)^3 < \infty. 
\eqe 

To conclude the proof, we will now argue that a.s.\ every  $z\in \BB C$ for which there are three or more distinct $D_h$-geodesics from 0 to $z$ belongs to $\mcl Z_{t,s}^w$ for some $w\in\BB Q^2$ and some $t,s\in\BB Q \cap (0,\infty)$ with $t < s < D_h(0,w)$.
To see this, let $P^1,P^2,P^3$ be three distinct $D_h$-geodesics from 0 to $z$. 
By Lemma~\ref{lem-geo-distinct}, there is an $r_* \in (0,D_h(0,z))$ such that $P^1(r) , P^2(r), P^3(r)$ are distinct for each $r\in [r_* ,D_h(0,z))$. 
We choose $t,s \in \BB Q\cap (0,\infty)$ with $r_* < t < s < D_h(0,z)$ and $w \in \BB C\setminus \mcl B_s$ such that $z \in \BB C\setminus \mcl B_s^{w,\bullet}$. 
Then $P^1(t) , P^2(t) , P^3(t)$ are distinct and so $z\in \mcl Z_{t,s}^w$, as required.  
\end{proof}

\subsection{Each point is joined to the origin by at most three geodesics} 
\label{sec-four-geo}

\begin{prop} \label{prop-four-geo}
Almost surely, for each $z\in\BB C$ there are at most three distinct $D_h$-geodesics from 0 to $z$.
\end{prop}

In light of Proposition~\ref{prop-three-geo-countable}, it is not hard to believe that Proposition~\ref{prop-four-geo} is true.
Intuitively, a ``generic" point which is joined to 0 by at least three geodesics should in fact be joined to zero by \emph{exactly} three geodesics. 
Since there are only countably many points joined to zero by at least three geodesics, we expect that all of them should be in some sense generic. 
The proof of Proposition~\ref{prop-four-geo} proceeds by making this intuition precise. The main step is the following lemma.

\begin{figure}[t!]
 \begin{center}
\includegraphics[scale=.9]{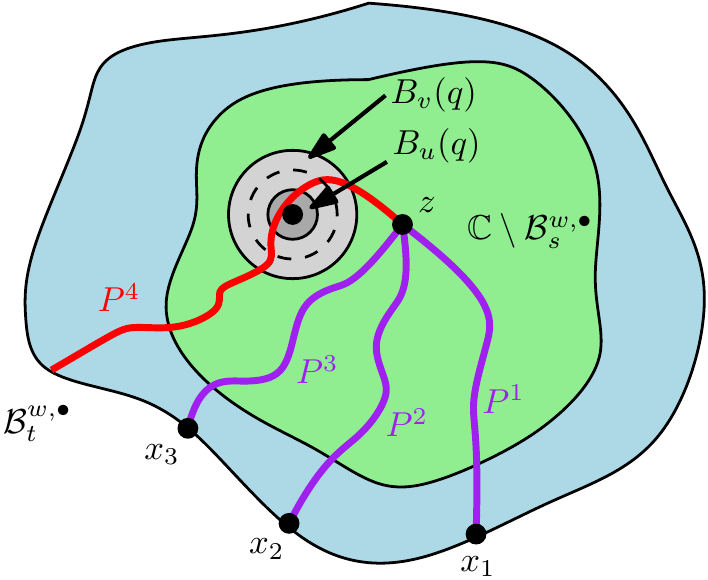}
\vspace{-0.01\textheight}
\caption{Illustration of the proof of Lemma~\ref{lem-four-geo-pts}. On $\wt E$, there are four $D_h$-geodesics $P^1,P^2,P^3,P^4$ from 0 to $z$ such that $P^4$ enters $B_u(q)$ but $P^1,P^2,P^3$ are disjoint from $B_v(q)$. If we add to $h$ a positive multiple of the smooth bump function $\phi$, which is supported on $B_v(q)$ and identically equal to 1 in a neighborhood of $B_u(q)$, then by Weyl scaling $P^1,P^2,P^3$ remain geodesics but $P^4$ does not. This means that $\wt E$ no longer occurs. Since adding a multiple of $\phi$ affects the law of $h$ in an absolutely continuous way, this shows that $\BB P[\wt E ] = 0$. 
}\label{fig-four-geo}
\end{center}
\vspace{-1em}
\end{figure}

\begin{lem}  \label{lem-four-geo-pts}
For each fixed $w\in\BB C$ and $0 < t < s$, the following is true a.s.\ on the event $\{s < D_h(0,w) \}$. 
Let $\mcl X = \mcl X_{t,s}^w \subset\bdy\mcl B_t^{w,\bullet}$ be the set of confluence points as in Proposition~\ref{prop-conf-finite} and let $x_1,x_2,x_3 \in \mcl X $ be distinct. 
There are no points $z \in \BB C\setminus \mcl B_s^{w,\bullet}$ with the following property: there are four distinct $D_h$-geodesics $P^1,P^2,P^3 , P^4$ from 0 to $z$ such that $P^1,P^2,P^3$ pass through $x_1,x_2,x_3$, respectively. 
\end{lem}

The idea of the proof of Lemma~\ref{lem-four-geo-pts} is that the event that there are four distinct $D_h$-geodesics $P^1,P^2,P^3,P^4$ from 0 to $z$ is unstable under small perturbations of $h$, in the following sense. 
If such geodesics exist with positive probability, then we can choose a smooth bump function $\phi : \BB C\rta [0,1]$ such that with positive probability,  $P^4$ intersects the support of $\phi$ but $P^1,P^2,P^3$ do not. On the event that this is the case, if we add any positive multiple of $\phi$ to $h$, then by Weyl scaling (Axiom~\ref{item-metric-f}) we will increase the length of $P^4$ while leaving the lengths of $P^1,P^2,P^3$ fixed. Hence $P^4$ will no longer be a geodesic. 
Since the laws of $h$ and $h+a\phi$ are mutually absolutely continuous for any $a \in \BB R$, this instability property is enough to conclude the lemma statement. 
See Figure~\ref{fig-four-geo} for an illustration.

\begin{proof}[Proof of Lemma~\ref{lem-four-geo-pts}]
The set of confluence points $\mcl X$ is $\sigma (\mcl B_s^{w,\bullet} , h|_{\mcl B_s^{w,\bullet}})$-measurable, so we can choose $x_1,x_2,x_3\in \mcl X$ in a $\sigma (\mcl B_s^{w,\bullet} , h|_{\mcl B_s^{w,\bullet}})$-measurable manner. 
Henceforth fix such a choice of $x_1,x_2,x_3$. 

Let $E$ be the event that the following is true.
\begin{enumerate}
\item There is a unique point $z \in \BB C\setminus \mcl B_s^{w,\bullet}$ with the following property: there are three distinct $D_h$-geodesics $P^1,P^2,P^3  $ from 0 to $z$ such that $P^1,P^2,P^3$ pass through $x_1,x_2,x_3$, respectively.  \label{item-four-geo-three}
\item For this point $z$, there is a $D_h$-geodesic $P^4$ from 0 to $z$ which is distinct from each of $P^1,P^2,P^3$. \label{item-four-geo-extra}
\end{enumerate}
In other words, $E$ is the event of the lemma statement except that we impose a uniqueness condition on $z$ in condition~\ref{item-four-geo-three}. 
By Lemma~\ref{lem-three-geo-pts}, a.s.\ there is at most one point satisfying the property of condition~\ref{item-four-geo-three}. 
Therefore, to prove the lemma statement we only need to show that $\BB P[E] = 0$. 
\medskip

\noindent\textit{Step 1: reducing to an event with extra conditions.}
We first reduce to proving that a secondary event, which we call $\wt E$, has probability zero. 
Consider a point $q \in \BB Q^2  $ and rational radii $u,v \in \BB Q$ with $0 < u < v$. 
We let $\wt E = \wt E(q,u,v)$ be the event that $E$ occurs and the following extra conditions are satisfied. 
\begin{enumerate} \setcounter{enumi}{2}
\item Each of the $D_h$-geodesics $P^1,P^2,P^3$ is disjoint from $B_v(q)$. 
\item There is a $D_h$-geodesic $P^4$ as in condition~\ref{item-four-geo-extra} in the definition of $E$ which enters $B_u(q)$. 
\item $B_v(q) \subset \BB C\setminus \mcl B_s^{w,\bullet}$.
\end{enumerate} 

We claim that if $E$ occurs, then a.s.\ $\wt E$ occurs for some choice of $q,u,v$ as above. 
To see this, assume that $E$ occurs.
By Lemma~\ref{lem-geo-distinct} a.s.\ there exists $r_* \in (0,D_h(0,z))$ such that the segments $P^i|_{[r_* , D_h(0,z))}$ for $i\in \{1,2,3,4\}$ are disjoint.
Therefore, if $r \in ( s \vee r_* , D_h(0,z))$, then the point $P^4(r) \in \BB C\setminus \mcl B_s^{w,\bullet}$ lies at positive distance from $P^i([r_*,D_h(0,z)])$ for each $i\in \{1,2,3 \}$. 
The point $P^4(r)$ also lies at positive distance from $\bdy\mcl B_s^{w,\bullet}$ and from $P^i([0,r_*])$ for each $i \in \{1,2,3,4\}$ since $\bdy\mcl B_s^{w,\bullet} , P^i([0,r_*]) \subset \mcl B_{r_*}(0;D_h)$ but $D_h(0,P^4(r)) = r > r_*$. 
Consequently, we can choose $q,u,v,$ as above such that $P^4(r) \in B_u(q)$, $B_v(q)$ is disjoint from $P^1,P^2,P^3$, and $B_v(q) \subset \BB C\setminus \mcl B_s^{w,\bullet}$. 
Hence $\wt E(q,u,v)$ occurs.  
\medskip

\noindent\textit{Step 2: reducing to a comparison between different fields.}
Since there are only countably many possible choices of the triple $(q,u,v)$, it suffices to fix one such triple and show that for this choice, $\BB P[\wt E ] = 0$. 
To show this, we will, roughly speaking, show that if $\wt E$ occurs and we add any positive multiple of a suitably chosen smooth bump function to $h$, then $\wt E$ no longer occurs. Since adding a smooth bump function affects the law of $h$ in an absolutely continuous way, this will tell us that $\BB P[\wt E ] = 0$. 

Let $u' := (u+v)/2$ and let $\phi : \BB C\rta [0,1]$ be a smooth bump function which is identically equal to 1 on $B_{u'}(q)$ and which vanishes outside of $B_v(q)$. 
By Weyl scaling (Axiom~\ref{item-metric-f}), on the event $\{ B_q(v) \subset \BB C\setminus \mcl B_s^{w,\bullet}\}$, a.s.\ adding a multiple of $\phi$ to $h$ does not affect $\mcl B_s^{w,\bullet}$ or $h|_{\mcl B_s^{w,\bullet}}$. In particular, on this event adding a multiple of $\phi$ to $h$ does not affect the values of $x_1,x_2,x_3$. 
Furthermore, for any $a \geq 0$ the laws of $h$ and $h+a\phi$ are mutually absolutely continuous.\footnote{By~\cite[Proposition 2.9]{ig4}, the laws of $h$ and $h+\phi$ are mutually absolutely continuous modulo additive constant. It is easy to see from the discussion above that we can arrange so that $B_v(q)$ is disjoint from $\bdy\BB D$. Then $h$ and $h+a\phi$ each have mean zero over $\bdy\BB D$, so their laws are mutually absolutely continuous not just modulo additive constant.}
Therefore, the following definition makes sense. 
For $a\geq 0$, let $\wt E_a$ be the event that $\wt E$ occurs with $h$ replaced by the field $h+a \phi$. 

We claim that, a.s., 
\eqb \label{eqn-four-geo-claim}
\text{If $0 \leq a < b$ and $\wt E_a$ occurs, then $\wt E_b$ does not occur.}
\eqe
The claim~\eqref{eqn-four-geo-claim} implies that a.s.\ there is only one value of $a \geq 0$ for which $\wt E_a$ occurs.
Consequently, if we sample a random variable $A$ uniformly at random from $[0,1]$, independently from $h$, then $\BB P[\wt E_A] = 0$. 
Since the laws of $h$ and $h+A\phi$ are mutually absolutely continuous, this implies that $\BB P[\wt E] =0$.
\medskip

\noindent\textit{Step 3: comparing $h+a\phi$ and $h+b\phi$.}
It remains only to prove the claim~\eqref{eqn-four-geo-claim}. In what follows, all a.s.\ statements are required to hold for \emph{all} values of $a$ and $b$. 
Assume that $0\leq a < b$ and $\wt E_a$ occurs. 
Let $z_a$ and $P_a^i$ for $i\in \{1,2,3 \}$ be as in the definition of $\wt E_a$ (i.e., the definition of $\wt E$ but with $h+a\phi$ in place of $h$). 

By Weyl scaling (Axiom~\ref{item-metric-f}) a.s.\ $D_{h+b\phi} \geq D_{h+a\phi}$. 
Since $P_a^1,P_a^2,$ and $P_a^3$ are disjoint from the support of $\phi$, the $D_{h+b\phi}$-lengths of each of these three paths are the same as their $D_{h+a\phi}$-lengths.
Therefore, each of $P_a^1,P_a^2$, and $P_a^3$ is also a $D_{h+b\phi}$-geodesic. 
In particular, $D_{h+b\phi}(0,z_a) = D_{h+a\phi}(0,z_a)$ and the point $z_a$ satisfies the property of condition~\ref{item-four-geo-three} in the definition of $E$ with $h+b\phi$ in place of $h$. 

Due to the uniqueness part of condition~\ref{item-four-geo-three}, in order to show that $\wt E_b$ does not occur it remains to show that there is no $D_{h+b\phi}$-geodesic from 0 to $z_a$ which enters $B_u(q)$. Since $D_{h+b\phi}(0,z_a) = D_{h+a\phi}(0,z_a)$, any $D_{h+b\phi}$-geodesic from 0 to $z_a$ has $D_{h+b\phi}$-length equal to $D_{h+a\phi}(0,z)$. From this and the fact that $D_{h+b\phi} \geq D_{h+a\phi}$, we infer that a $D_{h+b\phi}$-geodesic from 0 to $z_a$ is also a $D_{h+a\phi}$-geodesic. 

Let us therefore consider a $D_{h+a\phi}$-geodesic $P_a^4$ from 0 to $z_a$ which enters $B_u(q)$ and show that $P_a^4$ cannot be a $D_{h+b\phi}$-geodesic.
Write $\op{len}_{h+a\phi}$ and $\op{len}_{h+b\phi}$ for $D_{h+a\phi}$ and $D_{h+b\phi}$ lengths, respectively. 
Since $\phi$ is identically equal to 1 on $B_{u'}(q)$ and is non-negative, by Weyl scaling (Axiom~\ref{item-metric-f}), a.s.,
\alb
\op{len}_{h+b\phi}(P_a^4)
&\geq \op{len}_{h+b\phi}(P_a^4\cap B_{u'}(q)) + \op{len}_{h+b\phi}(P_a^4\setminus B_{u'}(q)) \notag\\
&\geq e^{\xi (b-a)} \op{len}_{h+a\phi}(P_a^4\cap B_{u'}(q)) + \op{len}_{h+a\phi}(P_a^4\setminus B_{u'}(q)) .
\ale
Since $P_a^4$ intersects $B_u(q)$, it must spend a positive amount of time in $B_{u'}(q)$ so by the previous display $\op{len}_{h+b\phi}(P_a^4)  > \op{len}_{h+a\phi}(P_a^4) = D_{h+a\phi}(0,z_a) = D_{h+b\phi}(0,z_a)$. 
Therefore $P_a^4$ is not a $D_{h+b\phi}$-geodesic and so $E_b$ does not occur.
\end{proof}

\begin{proof}[Proof of Proposition~\ref{prop-four-geo}]
By Lemma~\ref{lem-four-geo-pts} and a union bound over countably many possibilities for $w,t,s,$ a.s.\ for each $w\in \BB Q^2$ and each $t,s \in \BB Q\cap (0,\infty)$ such that $0 < t < s < D_h(0,w)$ the following is true.  
Let $\mcl X = \mcl X_{t,s}^w \subset\bdy\mcl B_t^{w,\bullet}$ be the set of confluence points as in Proposition~\ref{prop-conf-finite} and let $x_1,x_2,x_3 \in \mcl X $ be distinct. 
There are no points $z \in \BB C\setminus \mcl B_s^{w,\bullet}$ with the following property: there are four distinct $D_h$-geodesics $P^1,P^2,P^3 , P^4$ from 0 to $z$ such that $P^1,P^2,P^3$ pass through $x_1,x_2,x_3$, respectively (equivalently, $P^i(t) = x_i$ for each $i=1,2,3$). 

By Proposition~\ref{prop-conf-endpt-finite}, a.s.\ every $D_h$-geodesic from 0 to a point of $\BB C\setminus \mcl B_s^{w,\bullet}$ passes through some point of $\mcl X$, so the preceding paragraph implies that a.s.\ there are no points $z \in \BB C\setminus \mcl B_s^{w,\bullet}$ with the following property: there are four distinct $D_h$-geodesics $P^1,P^2,P^3 , P^4$ from 0 to $z$ such that $P^1(t) , P^2(t),P^3(t),P^4(t)$ are distinct. 

Now consider a point $z\in \BB C$ and four distinct $D_h$-geodesics $P^1,P^2,P^3,P^4$ from 0 to $z$. 
By Lemma~\ref{lem-geo-distinct}, on the event that such a point $z$ exists there is an $r_* \in (0,D_h(0,z))$ such that $P^1(r) , P^2(r), P^3(r) , P^4(r)$ are distinct for each $r\in [r_* ,D_h(0,z))$. 
Choose $t,s \in \BB Q\cap (0,\infty)$ with $r_* < t < s < D_h(0,z)$ and $w \in (\BB C\setminus \mcl B_s(0;D_h))\cap\BB Q^2$ such that $z \in \BB C\setminus \mcl B_s^{w,\bullet}$. 
Then $P^1(t) , P^2(t) , P^3(t) , P^4(t)$ are distinct, so from the preceding paragraph we conclude that the probability that such a point $z$ exists is zero. 
\end{proof}

\subsection{Existence of points joined to the origin by two or three geodesics}
\label{sec-three-geo-dense}

\begin{prop} \label{prop-three-geo-dense}
Let $\mcl R_{\geq 2}$ (resp.\ $\mcl R_{\geq 3}$) be the set of $z \in \BB C$ such that there are at least two (resp.\ three) distinct $D_h$-geodesics from 0 to $z$.  
Then every open subset of $\BB C$ contains at least one point of $\mcl R_{\geq 3}$ and an uncountable connected subset of $\mcl R_{\geq 2}$. 
\end{prop}

The main difficulty in the proof of Proposition~\ref{prop-three-geo-dense} is showing that $\mcl R_{\geq 3}\not=\emptyset$ and $\mcl R_{\geq 2}$ contains a connected set. This is accomplished via the following lemma.

\begin{figure}[t!]
 \begin{center}
\includegraphics[scale=.9]{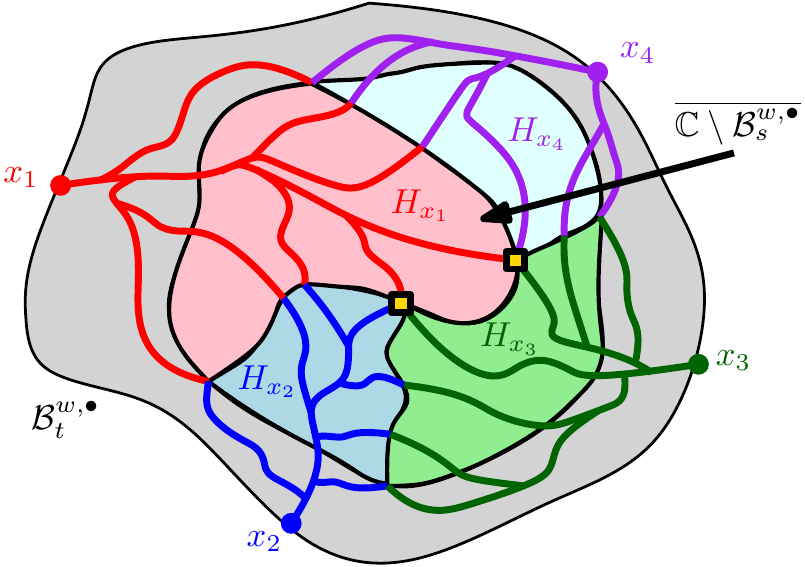}
\vspace{-0.01\textheight}
\caption{Illustration of the proof of Lemma~\ref{lem-three-geo-dense-pt}. Here, $\#\mcl X = 4$. We have shown each of the sets $H_x \in \mcl X$ in a different color and also shown several representative geodesics going from $x$ to points in $H_x$. Each point in $H_x \cap \bigcap_{y\not= x} H_y$ is joined to zero by at least two geodesics. Each point of $H_{x_1}\cap H_{x_2} \cap H_{x_3}$ for distinct $x_1,x_2,x_3 \in \mcl X$ is joined to zero by at least three geodesics. In the figure there are two points of this latter type, which are indicated by gold squares. 
}\label{fig-three-geo-dense}
\end{center}
\vspace{-1em}
\end{figure}

\begin{lem} \label{lem-three-geo-dense-pt}
Let $w\in\BB C\setminus \{0\}$ and $0 < t < s < \infty$. 
Define the set of confluence points $\mcl X = \mcl X_{t,s}^w \subset\bdy\mcl B_t^{w,\bullet}$ as in Proposition~\ref{prop-conf-finite}.
If $s < D_h(0,w)$, then a.s.\ for each $x\in\mcl X$ the following is true.
\begin{enumerate}
\item If $\#\mcl X \geq 3$, then there is a point $z\in \ol{\BB C\setminus \mcl B_s^{w,\bullet}}$ for which there are at least three distinct $D_h$-geodesics from 0 to $z$, one of which passes through $x$. \label{item-three-dense}
\item Suppose there is a point of $\BB C\setminus \mcl B_s^{w,\bullet}$ which is joined to zero by a geodesic which passes through $x$ and a point in $\BB C\setminus \mcl B_s^{w,\bullet}$ which is joined to zero by a geodesic which does not pass through $x$. Then there is an uncountable connected set $A\subset \ol{\BB C\setminus \mcl B_s^{w,\bullet}}$ such that for each $z\in A$, there are at least two distinct $D_h$-geodesics from 0 to $z$, one of which passes through $x$. \label{item-two-dense}
\end{enumerate}
\end{lem}
\begin{proof}
For $x \in \mcl X$, let $H_x$ be the set of points $z\in  \ol{\BB C\setminus \mcl B_s^{w,\bullet}}$ such that there is a $D_h$-geodesic from 0 to $z$ which passes through $x$.
Equivalently, 
\eqb \label{eqn-geo-cell-def}
H_x = \left\{z \in \ol{\BB C\setminus \mcl B_s^{w,\bullet}} : D_h(z,x) = D_h(z,\bdy\mcl B_s^{w,\bullet}) \right\} .
\eqe
To prove assertion~\ref{item-three-dense}, we will show that there must be distinct $y,y' \in \mcl X \setminus \{x\}$ such that $H_x\cap H_{y} \cap H_{y'} \not=\emptyset$. 
To prove assertion~\ref{item-two-dense}, we will show that $H_x \cap \bigcup_{y\in \mcl X\setminus \{x\}} H_y$ is uncountable and connected.
The proofs of both of these assertions are via deterministic topological arguments. 
See Figure~\ref{fig-three-geo-dense} for an illustration. 
\medskip

\noindent\textit{Step 1: the $H_x$'s cover $\ol{\BB C\setminus \mcl B_s^\bullet}$ and intersect only along their boundaries.}
By Proposition~\ref{prop-conf-endpt-finite}, a.s.\ every point of $\mcl B_s^{w,\bullet}$ belongs to $H_x$ for some $x\in\mcl X$. Furthermore, every point of $I_x$ except possibly the endpoints of $I_x$ belongs to $H_x$. 
By~\eqref{eqn-geo-cell-def}, each $H_x$ is closed so in fact $H_x$ contains $\ol I_x$. Consequently, 
\eqb \label{eqn-geo-cell-union}
 \ol{\BB C\setminus \mcl B_s^{w,\bullet}} = \bigcup_{x\in\mcl X} H_x .
\eqe

If $z\in H_x \cap H_y$ for $x \not=y$, then there are at least two distinct $D_h$-geodesics from 0 to $z$. By the uniqueness of $D_h$-geodesics to rational points, a.s.\ $H_x\cap H_y$ has empty interior for any $x\not=y$, so the sets $H_x$ intersect only along their boundaries. 
\medskip

\noindent\textit{Step 2: $H_x$ and its complement are path connected.}
For each $z\in H_x$, there is a $D_h$-geodesic from 0 to $z$ which passes through $x$. By Proposition~\ref{prop-conf-endpt-finite} this $D_h$-geodesic hits $\ol I_x$. 
Since $\ol I_x$ is a curve (see Lemma~\ref{lem-metric-curve}) it follows that each $H_x$ is connected.

If $z \in \ol{\BB C\setminus \mcl B_s^{w,\bullet}} \setminus H_x$, then there is no geodesic from 0 to $z$ which passes through $x$. Consequently, if $P$ is a geodesic from 0 to $z$, then $P$ cannot enter $H_x$ (otherwise, there would be a geodesic from 0 to some point of $P$ which passes through $x$, and hence also a geodesic from 0 to $z$ which passes through $x$).
In particular, $P$ must pass through $ \bdy\mcl B_s^{w,\bullet} \setminus \ol I_x$. The set $ \bdy\mcl B_s^{w,\bullet} \setminus \ol I_x$ is path connected (it is a union of arcs which intersect at their endpoints) and we have just shown that every point in $\ol{\BB C\setminus \mcl B_s^{w,\bullet}} \setminus H_x$ can be joined to this set by a path in $\ol{\BB C\setminus \mcl B_s^{w,\bullet}} \setminus H_x$.
Therefore, $\ol{\BB C\setminus \mcl B_s^{w,\bullet}} \setminus H_x$ is path connected.
\medskip

\noindent\textit{Step 3: the relative boundary of $H_x$ is connected.}
For $x\in\mcl X$, let $\wt\bdy H_x$ be the boundary of $H_x$ viewed as a subset of $\ol{\BB C\setminus \mcl B_s^{w,\bullet}}$, i.e.,
\eqbn
\wt\bdy H_x = \bdy H_x \cap \bdy (\ol{\BB C\setminus \mcl B_s^{w,\bullet}} \setminus H_x ) .
\eqen
We claim that for each $x\in\mcl X$, the set $\wt\bdy H_x$ is connected.
 
If $\BB C \setminus \mcl B_s^{w,\bullet}$ is bounded, then by Lemma~\ref{lem-metric-curve} and the discussion just after, the set $\ol{\BB C \setminus \mcl B_s^{w,\bullet}}$ is homeomorphic to a closed Euclidean disk, so its first homology group is trivial.
By Lemma~\ref{lem-connected-bdy} (applied with $K = H_x$) and the result of Step 2, we get that $\wt\bdy H_x$ is connected.
 
We now argue that the same is true if $\BB C\setminus \mcl B_s^{w,\bullet}$ is unbounded. Indeed, Lemma~\ref{lem-metric-curve} shows that $W := \ol{\BB C\setminus \mcl B_s^{w,\bullet}} \cup \{\infty\}$, viewed as a topological space with the one-point compactification topology, is homeomorphic to a closed Euclidean disk. Furthermore, by~\cite[Lemma 4.3]{gm-confluence} a.s.\ only one of the sets $H_x$ for $x\in\mcl X$ is unbounded. In particular, for each $x\in\mcl X$, $\wt\bdy H_x$ is bounded and coincides with the boundary of $H_x$ viewed as a subset of $W$. Therefore, we can apply Lemma~\ref{lem-connected-bdy} as above to get that $\wt\bdy H_x$ is connected.
\medskip

\noindent\textit{Step 4: proof of assertion~\ref{item-three-dense}.}
By~\eqref{eqn-geo-cell-union}, for each $x\in\mcl X$ the set $\wt\bdy H_x$ is the union of the closed sets $H_x \cap H_y$ for $y\in\mcl X\setminus \{x\}$. 
Since $\#\mcl X\geq 3$, at least two of the sets $H_x\cap H_y$ for $y\in\mcl X\setminus \{x\}$ must be non-empty, namely the ones for which $I_y$ and $I_x$ share an endpoint.
The union of all of the sets $H_x\cap H_y$ for $y\in \mcl X\setminus \{x\}$ is equal to $\wt\bdy H_x$, which by Step 3 is connected.
Since the sets $H_x\cap H_y$ are closed, these sets cannot all be disjoint.
Hence there must be distinct $y_1,y_2 \in \mcl X\setminus \{x\}$ for which $H_x\cap H_{y_1} \cap H_{y_2} \not=\emptyset$.
If $z\in H_x\cap H_{y_1}\cap H_{y_2}$, then there are $D_h$-geodesics from 0 to $z$ which pass through each of $x,y_1,y_2$. 
This gives assertion~\ref{item-three-dense}. 
\medskip

\noindent\textit{Step 5: proof of assertion~\ref{item-two-dense}.}
Assume the hypothesis of assertion~\ref{item-two-dense}, i.e., $H_x \cap (\BB C\setminus \mcl B_s^{w,\bullet}) \not=\emptyset$ and $H_y \cap (\BB C\setminus \mcl B_s^{w,\bullet}) \not=\emptyset$ for some $y\not= x$. If $P$ is a $D_h$-geodesic from 0 to a point $u \in H_x\cap  (\BB C\setminus \mcl B_s^{w,\bullet})$ which passes through $x$, then by Lemma~\ref{lem-regular-geo'} $P|_{[0,t]}$ is the unique $D_h$-geodesic from 0 to $P(t)$ for each $t\in (0,D_h(0,u))$. In particular, due to the last paragraph in Step 1, we must have $P(t) \notin \wt\bdy H_x$ for each $t\in (0,D_h(0,u))$. 
Hence $H_x$ has non-empty interior. Similarly, $H_y$ and hence also $\ol{\BB C\setminus \mcl B_s^{w,\bullet}} \setminus H_x$ has non-empty interior. 

As explained in Step 3, the set $\ol{\BB C\setminus \mcl B_s^{w,\bullet}}$ has the topology of either a closed disk or a closed disk with a point removed.
In particular, removing a single point from $\ol{\BB C\setminus \mcl B_s^{w,\bullet}}$ does not disconnect it. 
From the preceding paragraph, we see that $\ol{\BB C\setminus \mcl B_s^{w,\bullet}}\setminus \wt\bdy H_x$ is not connected, hence $\wt\bdy H_x$ is not a single point. 
Since $\wt\bdy H_x$ is connected (Step 3) it must be uncountable. 

Since each point of $\wt\bdy H_x$ is joined to zero by at least two distinct $D_h$-geodesics, at least one of which passes through $x$ (Step 1), we obtain assertion~\ref{item-two-dense} with $A = \wt \bdy H_x$. 
\end{proof}

We now prove an analog of Proposition~\ref{prop-three-geo-dense} restricted to points of $\bdy\mcl B_t^{w,\bullet}$. 

\begin{lem} \label{lem-three-geo-dense-bdy}
Let $w\in\BB C\setminus \{0\}$ and $t > 0$.
Almost surely on the event $\{t <D_h(0,w)\}$, each neighborhood of each point of $\bdy\mcl B_t^{w,\bullet}$ contains at least one point of $\mcl R_{\geq 3}$ and an uncountable connected subset of $\mcl R_{\geq 2}$. 
\end{lem}
\begin{proof}
For $s > t$, let $\mcl X_s = \mcl X_{t,s}^w $ be as in Proposition~\ref{prop-conf-finite}. 
Fix $z \in \bdy\mcl B_t^{w,\bullet}$ and $\ep > 0$. We will show that a.s.\ $ \mcl R_{\geq 3} \cap \mcl B_\ep(z;D_h) \not=\emptyset$ and that $\mcl R_{\geq 2} \cap \mcl B_\ep(z;D_h ) $ contains an uncountable connected set. See Figure~\ref{fig-three-geo-dense-bdy} for an illustration of the proof.

\begin{figure}[t!]
 \begin{center}
\includegraphics[scale=.9]{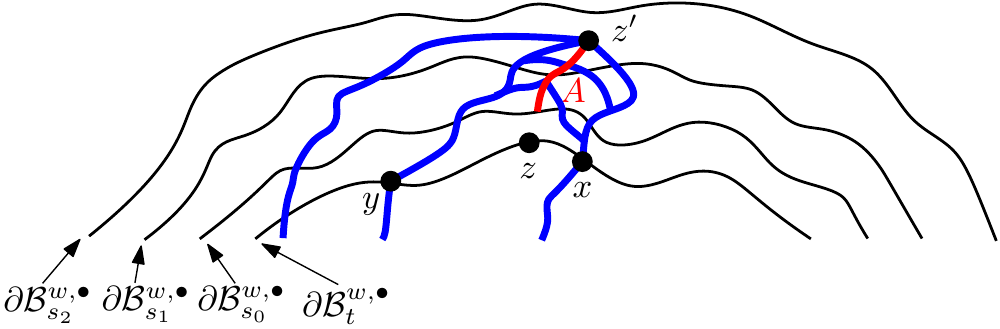}
\vspace{-0.01\textheight}
\caption{Illustration of the proof of Lemma~\ref{lem-three-geo-dense-bdy} (we just show a neighborhood of $z$). The point $x$ is chosen so that $x\in \mcl X_{t,s_1}^w \cap \mcl B_{\ep/2}(z;D_h)$. The point $z' \in \mcl B_{s_2}^{w,\bullet} \setminus \mcl B_{s_0}^{w,\bullet}$ is joined to 0 by three distinct geodesics, one of which passes through $x$. Each point in the connected set $A$ is joined to zero by two distinct $D_h$-geodesics, one of which passes through $x$. The point $z'$ and the set $A$ exist due to Lemma~\ref{lem-three-geo-dense-pt}. We show that $z'$ and $A$ are close to $z$ using that $x$ is close to $z$ and each point of $\mcl B_{s_2}^{w,\bullet}$ is close to $\mcl B_t^{w,\bullet}$. 
}\label{fig-three-geo-dense-bdy}
\end{center}
\vspace{-1em}
\end{figure}

We first argue that a.s.\
\eqb \label{eqn-infty-conf-pts}
\lim_{s \rta t^+} \#(\mcl X_s \cap \mcl B_\ep(z;D_h)) = \infty .
\eqe
Indeed, if $s>t$ and $u \in \bdy\mcl B_s^{w,\bullet}$, then every $D_h$-geodesic from 0 to $u$ must hit $\bdy\mcl B_t^{w,\bullet} \cap \ol{\mcl B_{s-t}(u;D_h)}$.
Hence any two $D_h$-geodesics to points of $\bdy\mcl B_s^{w,\bullet}$ which lie at $D_h$-distance greater than $2(s-t)$ from each other must hit different points of $\bdy\mcl B_t^{w,\bullet}$. 
Since the number of $D_h$-balls of radius $2(s-t)$ needed to cover $\bdy\mcl B_t^{w,\bullet} \cap \mcl B_\ep(z;D_h)$ tends to $\infty$ as $s\rta t^+$, we obtain~\eqref{eqn-infty-conf-pts}. 

Since $D_h$ induces the Euclidean topology and $\mcl B_t^{w,\bullet} = \bigcap_{s > t} \mcl B_s^{w,\bullet}$, we also have
\eqb \label{eqn-three-geo-diam}
\lim_{s\rta t^+} \max_{z' \in\mcl B_s^{w,\bullet}} D_h(z' ,\mcl B_t^{w,\bullet}) = 0 .
\eqe
 
By~\eqref{eqn-infty-conf-pts} and~\eqref{eqn-three-geo-diam}, there exists $s_1 , s_2  \in (t ,  t+\ep/2) \cap \BB Q$ with $s_1  <s_2$ such that 
\eqb \label{eqn-three-geo-times}
 \# [\mcl X_{s_1} \cap \mcl B_{\ep/2}(z;D_h) ] >   \#[\mcl X_{s_2 } \cap \mcl B_{\ep/2}(z;D_h)]  \geq 3 \quad \text{and} \quad
  \max_{z'\in\mcl B_{s_2}^{w,\bullet}} D_h(z',\mcl B_t^{w,\bullet}) \leq \ep/2 .
\eqe
Since $\mcl X_{s_1} \subset \mcl X_{s_2}$, the first condition in~\eqref{eqn-three-geo-times} implies that there is a point $x \in \mcl X_{s_1} \cap \mcl B_{\ep/2}(z;D_h)$ which is not in $\mcl X_{s_2}$. 

Let $s_0 \in (t,s_1) \cap \BB Q$. 
We now want to apply Lemma~\ref{lem-three-geo-dense-pt} with $s = s_0$ and the above choice of $x$. 
We have $x\in \mcl X_{s_0}$ and $\#\mcl X_{s_0} \geq \#\mcl X_{s_2} \geq 3$. 
Since $x\in\mcl X_{s_1}$ there must be a point in $\bdy\mcl B_{s_1}^{w,\bullet} \subset \BB C\setminus \mcl B_{s_0}^{w,\bullet}$ which is joined to zero by a $D_h$-geodesic which passes through $x$.
Since $\#\mcl X_{s_1} \geq 3$, there is some point $y\in \mcl X_{s_1}\setminus \{x\}$. As in the case of $x$, there is a point in $  \BB C\setminus \mcl B_{s_0}^{w,\bullet}$ which is joined to zero by a $D_h$-geodesic which passes through $y$.
Thus, the hypotheses of both assertions of Lemma~\ref{lem-three-geo-dense-pt} are satisfied, so Lemma~\ref{lem-three-geo-dense-pt} implies that a.s.\ the following is true. 
\begin{enumerate}
\item There is a point $z' \in \ol{\BB C\setminus \mcl B_{s_0}^{w,\bullet}}$ and at least three distinct $D_h$-geodesics from 0 to $z' $, one of which passes through $x$. 
\item There is an uncountable connected set $A\subset \ol{\BB C\setminus \mcl B_{s_0}^{w,\bullet}}$ such that for each $z''\in A$, there are at least two distinct $D_h$-geodesics from 0 to $z''$, one of which passes through $x$. 
\end{enumerate}

Since $x\notin \mcl X_{s_2}$, if $z'$ is as above then $z'  \in \ol{\mcl B_{s_2}^{w,\bullet} \setminus \mcl B_{s_0}^{w,\bullet}}$. 
By the second condition in~\eqref{eqn-three-geo-times}, $D_h(z', \bdy\mcl B_t^{w,\bullet}) \leq \ep/2$. 
Therefore, if $P$ is a $D_h$-geodesic from 0 to $z'$ which passes through $x$, then the $D_h$-diameter of $P([t,D_h(0,z')]$ is at most $\ep/2$. 
Consequently, $D_h(x,z') \leq \ep/2$ so since $x\in\mcl B_{\ep/2}(z;D_h)$, we have $z' \in \mcl B_\ep(z;D_h)$. 
Since $z' \in \mcl R_{\geq 3}$ by definition, we have $\mcl R_{\geq 3} \cap \mcl B_\ep(z;D_h) \not=\emptyset$. Similarly, we obtain that each of the points $z'' \in A$ above is contained in $\mcl B_\ep(z;D_h)$, so $ \mcl R_{\geq 2} \cap \mcl B_\ep(z;D_h) $ contains the uncountable connected set $A$. 
\end{proof}

\begin{proof}[Proof of Proposition~\ref{prop-three-geo-dense}]
The follows from  Lemma~\ref{lem-three-geo-dense-bdy} since the union of the sets $\bdy\mcl B_t^{w,\bullet}$ for $w\in\BB Q^2$ and $t\in \BB Q \cap (0,\infty)$ is a dense subset of $\BB C$.  
\end{proof}

\subsection{Proof of Theorem~\ref{thm-zero-geo}} 
\label{sec-zero-geo}

We prove all of the assertions of Theorem~\ref{thm-zero-geo} except that the Hausdorff dimension of the set of points joined to zero by two distinct geodesics is at most $d_\gamma-1$. This last assertion is Proposition~\ref{prop-two-geo-dim} below.

By Proposition~\ref{prop-four-geo}, a.s.\ for each $z\in\BB C$ there are either 1, 2, or 3 $D_h$-geodesics from 0 to $z$.
For each fixed $z\in\BB C$, a.s.\ there is a unique geodesic from 0 to $z$~\cite[Theorem 1.2]{mq-geodesics}. Hence the set of $z\in\BB C$ for which this is the case a.s.\ has full Lebesgue measure.

By Propositions~\ref{prop-three-geo-dense} and~\ref{prop-three-geo-countable}, a.s.\ the set $\mcl R_{\geq 3}$ of points $z\in\BB C$ for which there are at least (equivalently, exactly) three distinct geodesics from 0 to $z$ is dense and countable. 

By Proposition~\ref{prop-three-geo-dense}, if we let $\mcl R_{\geq 2}$ be the set of points $z\in \BB C$ such that there are at least two geodesics from 0 to $z$, then a.s.\ for any open set $U\subset \BB C$ the set $\mcl R_{\geq 2} \cap U$ contains an uncountable connected set. An uncountable connected metric space has Hausdorff dimension at least 1, so $\dim_{\mcl H}^\gamma (\mcl R_{\geq 2} \cap U) \geq 1$. 
Since $\mcl R_{\geq 3}$ is countable, it follows that the set $\mcl R_{\geq 2} \setminus \mcl R_{\geq 3}$ of points joined to zero by exactly two distinct geodesics is dense and has Hausdorff dimension at least 1. \qed

\subsection{Normal networks} 
\label{sec-normal-dense}

In this section we will prove Theorem~\ref{thm-dense}. The theorem is an easy consequence of Lemma~\ref{lem-regular-geo'}, Theorem~\ref{thm-zero-geo}, and the confluence of geodesics whose starting and ending points are close to ``typical" points (Theorem~\ref{thm-general-confluence}). 
The main input in the proof is the following proposition.

\begin{figure}[t!]
 \begin{center}
\includegraphics[scale=.9]{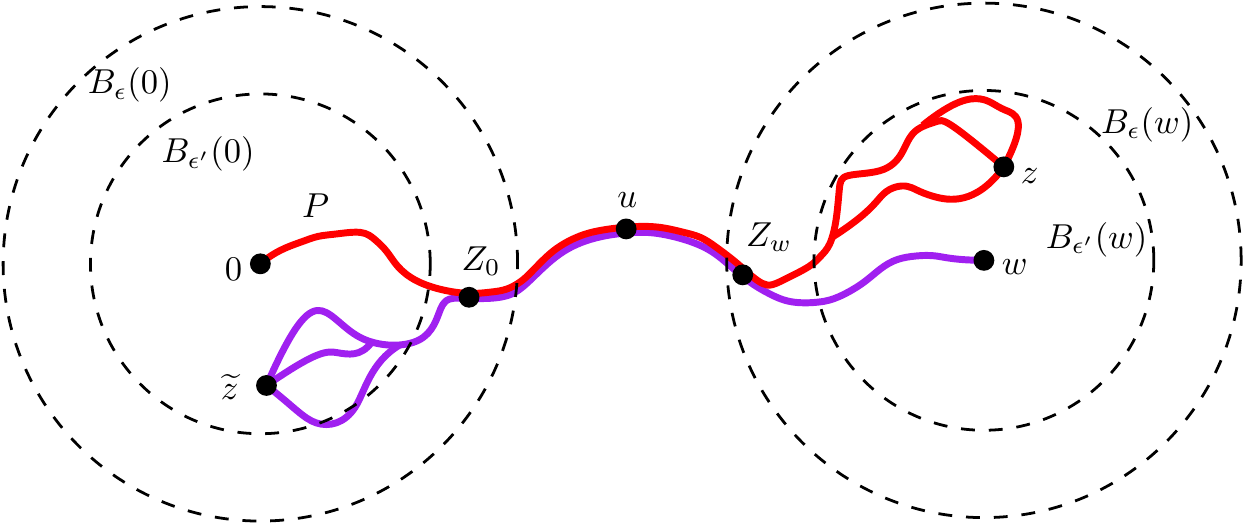}
\vspace{-0.01\textheight}
\caption{Illustration of the proof of the existence part of Proposition~\ref{prop-nine-geo-pts} in the case when $n=m=3$ (the proof that  $B_{\ep'}(0) \times B_{\ep'}(w) \subset \bigcup_{n,m \in \{1,2,3\}} N(n,m)$ uses the same setup). We apply Lemma~\ref{lem-regular-geo'} and Theorem~\ref{thm-zero-geo}, once for geodesics started from 0 and once for geodesics started $w$, to produce points $z \in B_{\ep'}(w)$ and $\wt z\in B_{\ep'}(0)$ such that $(0,z) \in N(1,m)$ and $(\wt z , w) \in N(n,1)$. Using confluence of geodesics, we can arrange that all of the geodesics from $0 $ to $z$ and all of the geodesics from $\wt z$ to $w$ coincide along the segment between the points $Z_0$ and $Z_w$. This implies that if $P$ is a geodesic from $0$ to $z$ and $P'$ is a geodesic from $\wt z$ to $w$, then the concatenation of the segment of $P'$ before it hits $u$ with the segment of $P$ after it hits $u$ is a $D_h$-geodesic from $\wt z$ to $z$. Consequently, $(\wt z ,z) \in N(n,m)$. 
}\label{fig-nine-geo}
\end{center}
\vspace{-1em}
\end{figure}

\begin{prop} \label{prop-nine-geo-pts}
Fix $w\in\BB C \setminus \{0\}$. Almost surely, for each $\ep > 0$ there exists $\ep' \in (0,\ep)$ such that $B_{\ep'}(0) \times B_{\ep'}(w) \subset \bigcup_{n,m \in \{1,2,3\}} N(n,m)$ and $(B_{\ep'}(0) \times B_{\ep'}(w) ) \cap N(n,m) \not=\emptyset$ for each $n,m \in \{1,2,3\}$. 
\end{prop}
\begin{proof}
The proof is similar to that of~\cite[Theorems 8 and 9]{akm-geodesics}. See Figure~\ref{fig-nine-geo} for an illustration. 
We can assume without loss of generality that $\ep \in (0,|w|/3)$, so $B_\ep(z) \cap B_\ep(w) \not=\emptyset$. 
By Theorem~\ref{thm-general-confluence}, a.s.\ there exists $\ep' \in (0,\ep)$ and $Z_0 \in B_\ep(0) \setminus B_{\ep'}(w)$ such that each $D_h$-geodesic from a point of $B_{\ep'}(0)$ to a point of $\BB C\setminus B_{\ep }(0)$ passes through $Z_0$. 
Symmetrically, by possibly shrinking $\ep'$ we can arrange that there is also a point $Z_w\in B_\ep(w) \setminus B_{\ep'}(w)$ such that each $D_h$-geodesic from a point of $B_{\ep'}(w)$ to a point of $\BB C\setminus B_{\ep }(w)$ passes through $Z_w$.  

Every $D_h$-geodesic from a point of $B_{\ep'}(0)$ to a point of $B_{\ep'}(w)$ hits $Z_0$, then $Z_w$. Since there is a.s.\ a unique $D_h$-geodesic from 0 to $w$, it follows that a.s.\ there is a unique $D_h$-geodesic $P_*$ from $Z_0$ to $Z_w$, and every $D_h$-geodesic from a point of $B_{\ep'}(0)$ to a point of $B_{\ep'}(w)$ must traverse $P_*$. 
Let $u$ be a point of $P_*$ which is not one of its endpoints. 
\medskip

\noindent\textit{Step 1: proof that $B_{\ep'}(0) \times B_{\ep'}(w) \subset \bigcup_{n,m \in \{1,2,3\}} N(n,m)$.}
By Lemma~\ref{lem-regular-geo'} and Proposition~\ref{prop-four-geo}, a.s.\ $(0,z) \in \bigcup_{m=1}^3 N(1,m)$ for each $z\in B_{\ep'}(w)$. 
Furthermore, from the preceding paragraph we know that any $D_h$-geodesic from 0 to $z$ traces $P_*$ for a positive interval of times after hitting $u$. 
If $z \in B_{\ep'}(w)$, then each geodesic $P$ from $u$ to $z$ gives rise to a distinct geodesic from 0 to $z$ by concatenating a geodesic from 0 to $u$ with the segment of $P$ after it hits $u$. This concatenation is a geodesic since every geodesic from 0 to $z$ hits $u$.  
It follows that $(u , z) \in \bigcup_{m=1}^3 N(1,m)$ for each $z \in B_{\ep'}(w)$.  
Similarly, $( \wt z , u) \in \bigcup_{n=1}^3 N(n,1)$ for each $\wt z \in B_{\ep'}(0)$.   
 
If $\wt z \in B_{\ep'}(0)$ and $z \in B_{\ep'}(w)$, then every $D_h$-geodesic from $\wt z$ to $z$ is the concatenation of a $D_h$-geodesic from $\wt z$ to $u$ and a $D_h$-geodesic from $u$ to $\wt z$. By the preceding paragraph, we get that $B_{\ep'}(0) \times B_{\ep'}(w)\subset \bigcup_{n,m \in \{1,2,3\}} N(n,m)$. 
\medskip

\noindent\textit{Step 2: proof that $(B_{\ep'}(0) \times B_{\ep'}(w) ) \cap N(n,m) \not=\emptyset$ for each $n,m \in \{1,2,3\}$.}
By Lemma~\ref{lem-regular-geo'} and Theorem~\ref{thm-zero-geo}, a.s.\ there exists $z\in B_{\ep'}(w)$ such that $(0, z) \in N(1,m)$. 
Each $D_h$-geodesic from 0 to $z$ hits $u$ then subsequently traces $P_*$ for a non-trivial interval of time. 
From this, we infer that $(u,z) \in N(1,m)$. 
Similarly, a.s.\ there exists $\wt z \in B_{\ep'}(0)$ such that $(\wt z , u) \in N(1,n)$. 
Since every $D_h$-geodesic from $\wt z$ to $z$ passes through $u$, we get that the set of $D_h$-geodesics from $\wt z$ to $z$ is precisely the set of paths obtained as the concatenation of a $D_h$-geodesic from $\wt z$ to $u$ and a $D_h$-geodesic from $u$ to $z$. 
Therefore, $(\wt z , z) \in N(n,m)$.  
\end{proof}

\begin{proof}[Proof of Theorem~\ref{thm-dense}]
Due to the translation invariance of the law of $h$, modulo additive constant, Proposition~\ref{prop-nine-geo-pts} implies that a.s.\ for each $(z,w) \in\BB Q^2 \times \BB Q^2$ and each $\ep > 0$, there exists $\ep' \in (0,\ep)$ such that $B_{\ep'}(z) \times B_{\ep'}(w) \subset \bigcup_{n,m \in \{1,2,3\}} N(n,m)$ and $(B_{\ep'}(z) \times B_{\ep'}(w) ) \cap N(n,m) \not=\emptyset$ for each $n,m \in \{1,2,3\}$. 
From this, the theorem statement is immediate. 
\end{proof}

\section{Finitely many geodesics between any two points}
\label{sec-multi-geo-dim}

The goal of this section is to prove Theorem~\ref{thm-finite-geo} (which gives an upper bound for the maximal number of distinct $D_h$-geodesics joining any two points in $\BB C$) and the remaining assertion of Theorem~\ref{thm-zero-geo}, namely the upper bound for the Hausdorff dimension of the set of points joined to zero by at least two distinct geodesics. 

The main input needed for this purpose is a bound for the size of the set of pairs of points in $\BB C$ which can be joined by a network of geodesics with a certain non-overlapping property, which we now define. 
For $n\in\BB N$, let $\mcl S_n$ be the set of pairs of points $(z,w) \in \BB C \times \BB C$ with the following property. 
There exist $n+1$ distinct geodesics $P^0,\dots,P^n$ from $z$ to $w$ such that the following is true. 
For each $i \in [1,n]_{\BB Z}$, there is a point $u_i \in \BB C$ such that $u_i $ is hit by $P^i$ and $u_i$ is not hit by $P^j$, $\forall j \in [0,n]_{\BB Z} \setminus \{i\}$.
Note that the case $i=0$ is somewhat special here: there is no point $u_0$, i.e., it is possible that every point on $P^0$ is hit by $P^j$ for some $j\not=0$. 
See Figure~\ref{fig-geo-overlap} for an illustration of the definition of $\mcl S_n$.

\begin{figure}[t!]
 \begin{center}
\includegraphics[scale=.9]{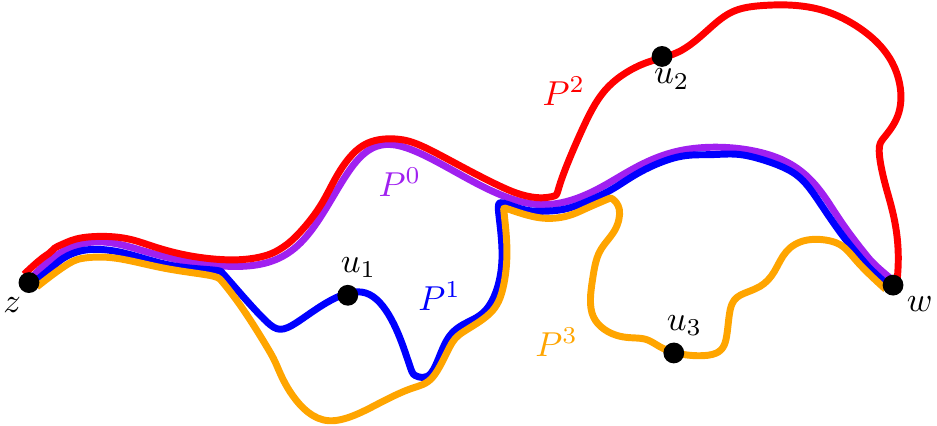}
\vspace{-0.01\textheight}
\caption{A pair of points $(z,w) \in \mcl S_3$ along with a possible choice of the corresponding geodesics $P^0,P^1,P^2,P^3$ and points $u_1,u_2,u_3$. 
The figure illustrates several important points about the definition on $\mcl S_n$. 
First, there are nine distinct $D_h$-geodesics from $z$ to $w$ (which can be obtained by concatenating different segments of the four colored geodesics).
However, $(z,w)\notin \mcl S_n$ for any $n\geq 4$ since we can only find four geodesics satisfying the non-overlapping condition in the definition of $\mcl S_n$. 
Second, the purple geodesic is contained in the union of the red and blue geodesics, so the purple geodesic must be $P^0$ (we can re-label the other three geodesics arbitrarily, however).  
}\label{fig-geo-overlap}
\end{center}
\vspace{-1em}
\end{figure}

\begin{prop} \label{prop-multi-geo-dim}
If $n\leq \lfloor 2d_\gamma  \rfloor$, then a.s.\ $\dim_{\mcl H}^\gamma \mcl S_n \leq 2d_\gamma - n$. 
If $n > \lfloor 2d_\gamma  \rfloor$, then a.s.\ $\mcl S_n = \emptyset$. 
\end{prop}

We note that $(z,w) \in \mcl S_1$ if and only if there are at least two distinct $D_h$-geodesics from $z$ to $w$. Therefore, Proposition~\ref{prop-multi-geo-dim} immediately implies the following corollary. 
 
\begin{cor} \label{cor-two-geo-dim}
Almost surely, the set of pairs of points $(z,w) \in \BB C \times \BB C$ such that $z$ and $w$ can be joined by at least 2 distinct $D_h$-geodesics has $\gamma$-LQG dimension at most $2d_\gamma-1$. 
\end{cor}

We do not expect that the upper bound in Corollary~\ref{cor-two-geo-dim} is optimal. 
For $n\geq 3$, there can be points $(z,w) \in \BB C\times \BB C$ such that $z$ and $w$ can be joined by at least $n$ distinct $D_h$-geodesics but $(z,w) \notin \mcl S_n$; see Figure~\ref{fig-geo-overlap} for an example.

Most of this section is devoted to the proof of Proposition~\ref{prop-multi-geo-dim}. 
In Section~\ref{sec-geo-finite}, we will explain how to use Proposition~\ref{prop-multi-geo-dim} to prove Theorem~\ref{thm-finite-geo} by reducing a given collection of geodesics with the same endpoints to a collection which satisfies the conditions in the definition of $\mcl S_n$. 
In Section~\ref{sec-two-geo-dim}, we will explain how to adapt the arguments used in the proof of Proposition~\ref{prop-multi-geo-dim} to get the Hausdorff dimension upper bound of Theorem~\ref{thm-zero-geo}.

We now give an overview of the proof of Proposition~\ref{prop-multi-geo-dim}. 
In Section~\ref{sec-reduction}, we fix a collection of open sets $U_1,V_1,\dots,U_n , V_n$ such that $\ol U_i \subset V_i$ and $V_i\cap V_j =\emptyset$ for each $i,j=1,\dots,n$. 
We then explain why it suffices to prove an analog of Proposition~\ref{prop-multi-geo-dim} with $\mcl S_n$ replaced by the set $\wt{\mcl S}_n = \wt{\mcl S}_n( U_1,\dots,U_n,V_1,\dots,V_n )$, which is defined in the same manner as $\mcl S_n$ except that the point $u^i \in P^i$ is required to belong to $U_i$ for each $i=1,\dots,n$ and the geodesics $P^j$ for $j\not=i$ are required to be disjoint from $V_i$. The reason why we can reduce to upper-bounding the dimension of $\wt{\mcl S}_n$ is that we can cover $\mcl S_n$ by countably many sets of the form $\wt{\mcl S}_n$ for varying choices of $U_1,V_1,\dots,U_n,V_n$. 

In Section~\ref{sec-geo-count-prob}, we will prove the ``one-point estimate" needed for the proof of the above analog of Proposition~\ref{prop-multi-geo-dim}. To this end, we will define for each $z,w\in \BB C$ and each $\ep > 0$ an event $E^\ep(z,w)$ which is, roughly speaking, the event that $\wt{\mcl S}_n$ intersects $\mcl B_\ep(z;D_h)\times \mcl B_\ep(w;D_h)$ (plus some regularity conditions). We then prove in Lemma~\ref{lem-geo-count-prob} that $\BB P[E^\ep(z,w) \cap G(C)] \leq \ep^{n  +o_n(1)}$, where $G(C)$ is a global regularity event depending on a parameter $C$. 
The proof of this estimate is based on the following perturbation argument, which is similar in flavor to the argument used in Section~\ref{sec-four-geo}, but more quantitative. For $i=1,\dots,n$, let $\phi_i : \BB C \rta [0,1]$ be a smooth bump function supported on $V_i$ and let $X_1,\dots,X_n$ be i.i.d.\ random variables with the uniform distribution on $[0,1]$. 
If $E^\ep(z,w)$ occurs, then with high probability adding the function $\sum_{i=1}^n X_i \phi_i$ to $h$ distorts distances enough that one of the $P^i$'s is no longer a geodesic. This makes it so that $E^\ep(z,w)$ does not occur with $h + \sum_{i=1}^n X_i \phi_i$ in place of $h$. We will then conclude our estimate using the fact that the laws of $h + \sum_{i=1}^n X_i \phi_i$ and $h$ are mutually absolutely continuous, with an explicit Radon-Nikodym derivative. 

In Section~\ref{sec-subset-dim}, we will deduce Proposition~\ref{prop-multi-geo-dim} from the above one-point estimate by covering a given bounded open subset of $\BB C$ by $\ep^{-d_\gamma + o_\ep(1)}$ LQG balls of radius $\ep$ and applying the one-point estimate to the center points of the balls (actually, we need a variant of the one-point estimate where the points $z$ and $w$ can be sampled from the $\gamma$-LQG measure, see Lemma~\ref{lem-random-pt}). The existence of the needed covering comes from Lemma~\ref{lem-ball-cover}, which in turn is a consequence of results from~\cite{afs-metric-ball}.

\subsection{Reducing to a bound for a fixed collection of sets}
\label{sec-reduction}

To prove Proposition~\ref{prop-multi-geo-dim}, we begin with some reductions. 
Suppose $(z,w) \in \mcl S_n$ and let $P^0,\dots,P^n$ and $u_1,\dots,u_n$ be geodesics and points, respectively, as in the definition of $\mcl S_n$. 
Since the trace of each $P^i$ is a closed set, we can find open sets $U_1,\dots,U_n , V_1,\dots,V_n \subset \BB C$, each of which is a finite union of Euclidean balls with rational centers and radii, such that
\eqb \label{eqn-multi-geo-set}
 \ol U_i\subset V_i   \quad \text{and} \quad \ol V_i\cap \ol V_j =\emptyset , \quad\forall i ,j \in [1,n]_{\BB Z} ,\quad i\not=j   ,
\eqe
and the following is true. For each $i \in [1,n]_{\BB Z}$, we have $u_i \in U_i$ (which implies that $P^i\cap U_i \not=\emptyset$) and for each $j\in [0,n]_{\BB Z}\setminus \{i\}$ we have $P^j \cap \ol V_i = \emptyset$. 
Moreover, by possibly replacing each $u_i$ with a nearby point on $P^i$ and subsequently shrinking each of the $U_i$'s and $V_i$'s, we can arrange that
\eqb \label{eqn-multi-geo-bdy}
V_i \cap \bdy\BB D = \emptyset, \quad \forall i\in [1,n]_{\BB Z} .
\eqe
We note that an LQG geodesic cannot trace an arc of $\bdy \BB D$; see, e.g.,~\cite[Proposition 4.1]{lqg-metric-estimates}. 
The condition~\eqref{eqn-multi-geo-bdy} is important for our purposes since $h$ is normalized so that its average over $\bdy\BB D$ is zero. 

Since we required that $U_1,\dots,U_n,V_1,\dots,V_n$ are finite unions of Euclidean balls with rational centers and radii, there are only countably many possible choices of $U_1,\dots,U_n,V_1,\dots,V_n$. By the countable stability of Hausdorff dimension, it therefore suffices to fix bounded open sets $U_1,\dots,U_n,V_1,\dots,V_n$ such that~\eqref{eqn-multi-geo-set} and~\eqref{eqn-multi-geo-bdy} hold; then establish an upper bound for the Hausdorff dimension of the set of pairs $(z,w) \in \BB C\times \BB C$ for which the above conditions on geodesics hold with the given choice of $U_1,\dots,U_n,V_1,\dots,V_n$. 

Let us now make this more precise. 
Let $\wt{\mcl S}_n = \wt{\mcl S}_n( U_1,\dots,U_n,V_1,\dots,V_n )$ be the set of pairs $(z,w) \in \BB C\times \BB C$ for which the following is true. 
There are $n+1$ distinct $D_h$-geodesics $P^0,\dots,P^n$ from $z$ to $w$ such that
\eqb \label{eqn-subset-def}
P^i \cap U_i \not=\emptyset \quad \text{and} \quad  P^i \cap \bigcup_{j \in [1,n]_{\BB Z} \setminus \{i\} } \ol V_j = \emptyset ,\quad \forall i \in [1,n]_{\BB Z} , \quad \text{and} \quad P^0 \cap \bigcup_{j \in [1,n]_{\BB Z}   } \ol V_j = \emptyset .
\eqe  
By the above discussion, to prove Proposition~\ref{prop-multi-geo-dim}, it suffices to prove the following.

\begin{prop} \label{prop-subset-dim}
Fix any collection of bounded open sets $U_1,\dots,U_n,V_1,\dots,V_n \subset\BB C$ satisfying~\eqref{eqn-multi-geo-set} and~\eqref{eqn-multi-geo-bdy}.  
If $n\leq \lfloor 2d_\gamma   \rfloor$, then a.s.\ $\dim_{\mcl H}^\gamma \wt{\mcl S}_n \leq 2d_\gamma - n$. 
If $n > \lfloor 2d_\gamma  \rfloor$, then a.s.\ $\wt{\mcl S}_n = \emptyset$. 
\end{prop}

The rest of this section is devoted to the proof of Proposition~\ref{prop-subset-dim}.

\subsection{One-point estimate}
\label{sec-geo-count-prob}
  
For $\ep > 0$ and $z,w\in \BB C$, let $  E^\ep(z,w)$ be the event that the following is true.
\begin{enumerate}
\item The sets $\mcl B_\ep(z; D_h) $, $\mcl B_\ep(w ; D_h)$, and $ \bigcup_{i=1}^n \ol V_i$ are disjoint.   
\item There are $D_h$-geodesics $P^0,\dots,P^n$ such that each $P^i$ goes from a point of $\mcl B_\ep(z ; D_h)$ to a point of $\mcl B_\ep(w ; D_h)$, 
\eqb \label{eqn-subset-def'}
P^i \cap U_i \not=\emptyset  \quad \text{and} \quad P^i \cap \bigcup_{j \in [1,n]_{\BB Z}\setminus \{i\} } \ol V_j = \emptyset ,\quad \forall i \in [1,n]_{\BB Z} ,
 \quad \text{and} \quad P^0 \cap \bigcup_{j \in [1,n]_{\BB Z}   } \ol V_j = \emptyset .
\eqe 
\end{enumerate}
See Figure~\ref{fig-geo-count} for an illustration.

\begin{figure}[t!]
 \begin{center}
\includegraphics[scale=.9]{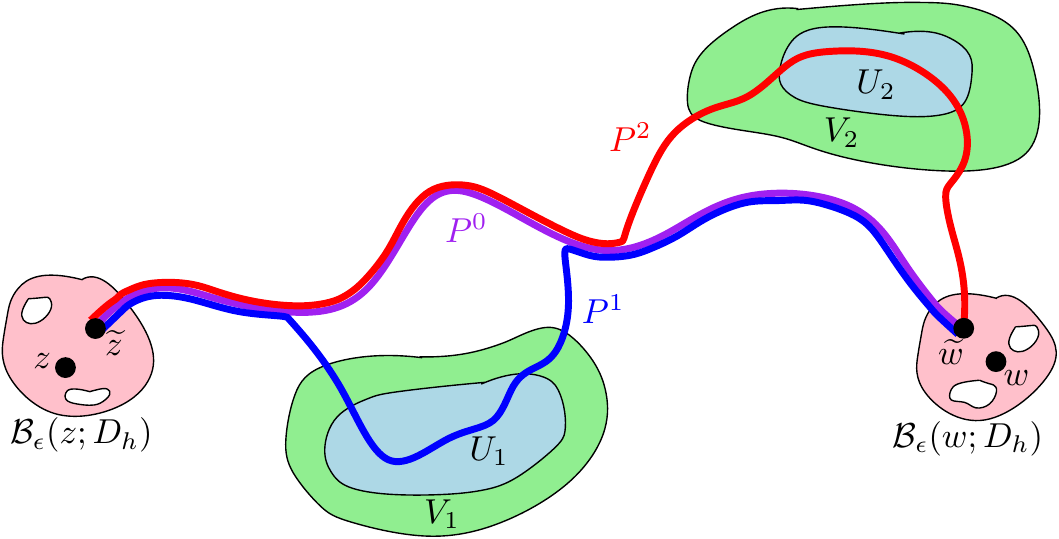}
\vspace{-0.01\textheight}
\caption{The figure shows a pair of points $(\wt z , \wt w) \in \wt{\mcl S}_2$ and the three corresponding $D_h$-geodesics $P^0, P^1,P^2$ from $\wt z$ to $\wt w$, as well as a pair of points $(z,w) $ for which $E^\ep(z,w)$ occurs.  
We emphasize that in the definition of $E^\ep(z,w)$ we do \emph{not} require that the geodesics go from $z$ to $w$, only that they each go from a point of $\mcl B_\ep(z;D_h)$ to a point of $\mcl B_\ep(w;D_h)$. 
In the figure, the starting and ending points of all three geodesics are the same, but this is not required.
}\label{fig-geo-count}
\end{center}
\vspace{-1em}
\end{figure}

The main step in the proof of Proposition~\ref{prop-subset-dim} is an upper bound for $\BB P[E^\ep(z,w)]$. 
The idea of the proof of this upper bound is, at a very rough level, similar to the proof of Proposition~\ref{prop-four-geo}. We will show that if $E^\ep(z,w)$ occurs and for each $i\in [1,n]_{\BB Z}$, we add to $h$ a smooth bump function which is supported in $V_i$ and takes values in $[0,A\ep]$ for a constant $A>0$, then $E^\ep(z,w)$ no longer occurs. 
Since adding these smooth bump functions affects the law of $h$ in an absolutely continuous way, this leads to an upper bound for $\BB P[E^\ep(z,w)]$. 

To make the above heuristic precise, we first define the particular bump functions we will consider. 
For $i\in [1,n]_{\BB Z}$, let $U_i'$ be an intermediate open set with $\ol U_i \subset U_i' $ and $\ol U_i' \subset V_i$. 
Let $\phi_i$ be a smooth bump function which is identically equal to 1 on $U_i'$ and is identically equal to 0 outside of $V_i$. 
For a vector $\bd x = (x_1,\dots,x_n ) \in [0,1]^n$, let 
\eqb
h_{\bd x} := h + \sum_{i=1}^n x_i \phi_i .
\eqe 
Also define $E_{\bd x}^\ep(z,w)$ as above but with $h_{\bd x}$ in place of $h$. 

We will then deduce Proposition~\ref{prop-subset-dim} from the following two facts.
\begin{enumerate}[(i)]
\item Suppose $\bd x , \bd y \in [0,1]^n$ and the $L^\infty$-distance between $\bd x$ and $\bd y$ is at least a constant times $\ep$.
If $E_{\bd x}^\ep(z,w)$ occurs, then $E_{\bd y}^\ep(z,w)$ does not occur (Lemma~\ref{lem-geo-count-perturb}). \label{item-geo-outline-perturb}
\item By basic estimates for the GFF, if $\bd X$ is sampled from Lebesgue measure on $[0,1]^n$ independently from $h$ then the laws of $h$ and $h_{\bd X}$ are mutually absolutely continuous. Moreover, the Radon-Nikodym derivative has finite moments of all positive and negative orders.  \label{item-geo-outline-rn}
\end{enumerate}
Roughly speaking, the reason why fact~\eqref{item-geo-outline-perturb} is true is that changing the value of $\bd x$ by at least a constant times $\ep$ will make it so that at least one of the geodesics $P^0,\dots,P^n$ in the definition of $E^\ep(z,w)$ is no longer length-minimizing. For this purpose, it is crucial that the geodesic $P^0$ is disjoint from the supports of all of the bump functions which we are adding to $h$, so that changing the value of $\bd x$ does not change the length of $P^0$. If changing $\bd x$ could change the lengths of all $n$ of the geodesics, then it could be that the changes to $\bd x$ conspire in such a way as to change the length of each of the geodesics by the same amount. 

Fact~\eqref{item-geo-outline-perturb} implies that the Lebesgue measure of the set of $\bd x\in [0,1]^n$ for which $E_{\bd x}^\ep(z,w)$ occurs is at most a constant times $\ep^n$. Equivalently, if $\bd X$ is as in fact~\eqref{item-geo-outline-rn}, then $\BB P[E_{\bd X}^\ep(z,w)] = O_\ep(\ep^n)$ (see Lemma~\ref{lem-geo-count-unif}). 
Combining this with the absolute continuity statement from fact~\eqref{item-geo-outline-rn} leads to the needed upper bound for $\BB P[E^\ep(z,w)]$. 

In order to ensure that changing the value of $\bd x$ has a large enough effect on the $D_{h_{\bd x}}$-lengths of paths, throughout most of the proof we will need to truncate on the global regularity event  
\eqb \label{eqn-geo-count-reg}
  G(C) := \left\{ D_h(\bdy U_i , \bdy U_i') \geq 1/C ,\: \forall i \in [1,n]_{\BB Z} \right\}  ,\quad \forall  C >1 .
\eqe 
We also define $G_{\bd x}(C)$ in the same manner as in~\eqref{eqn-geo-count-reg} but with $h_{\bd x}$ in place of $h$.
With this event in hand, we can now state the main estimate which goes into the proof of Proposition~\ref{prop-subset-dim}.

\begin{lem} \label{lem-geo-count-prob} 
For each $C>1$, each $z,w\in \BB C$, and each $\ep > 0$, 
\eqb \label{eqn-geo-count-prob}
\BB P\left[ E^\ep(z,w) \cap G(C) \right] \leq \ep^{n + o_\ep(1)}
\eqe
where here the rate of the $o_\ep(1)$ does not depend on $z$ or $w$. 
\end{lem}

The following elementary lemma is the reason for the first condition in the definition of $E^\ep(z,w)$. 
For the statement, we recall the number $\xi = \gamma/d_\gamma$ from~\eqref{eqn-xi-Q}. 

\begin{lem} \label{lem-geo-count-disjoint}
Suppose that for some $\bd x\in [0,1]^n$, the sets $\mcl B_\ep(z ; D_{h_{\bd x}}) $, $\mcl B_\ep(\BB w ; D_{h_{\bd x}})$, and $ \bigcup_{i=1}^n \ol V_i$ are disjoint. 
Then $\mcl B_\ep(z ; D_{h_{\bd y}}) = \mcl B_\ep(z ; D_{h_{\bd x}})$ and $\mcl B_\ep(w ; D_{h_{\bd y}})= \mcl B_\ep(z ; D_{h_{\bd x}})$ for every $\bd y \in [0,1]^n$. 
\end{lem}
\begin{proof}
Since each $\phi_i$ is supported on $V_i$, Weyl scaling (Axiom~\ref{item-metric-f}) implies that adding a linear combination of the $\phi_i$'s to $h_{\bd x}$ does not change the $D_{h_{\bd x}}$-length of any path which does not enter $ \bigcup_{i=1}^n \ol V_i$. From this observation, the lemma statement is immediate. 
\end{proof}

Our next lemma roughly corresponds to fact~\eqref{item-geo-outline-perturb} in the above outline.

\begin{lem} \label{lem-geo-count-perturb}
Fix $C>1$ and $z,w\in\BB C$. 
Let $\ep \in (0,1/(100 C) )$ and let $\bd x , \bd y \in  [0,1]^{n}$ such that $\max_{i\in [1,n]_{\BB Z}}|x_i - y_i| > 8 e \xi^{-1} C  \ep$. If $E^\ep_{\bd x}(z,w) \cap G(C)$ occurs, then $E^\ep_{\bd y}(z,w)$ does not occur.
\end{lem}
\begin{proof}
Suppose that $E^\ep_{\bd x}(z,w) \cap G (C)$ occurs. The proof is elementary and completely deterministic. 
Let $P^0_{\bd x}, \dots, P^n_{\bd x}$ be the $D_{h_{\bd x}}$-geodesics as in the definition of $E_{\bd x}^\ep(z,w)$. 
Let $i\in [1,n]_{\BB Z}$ be chosen so that $|x_i - y_i| > 8 e \xi^{-1} C  \ep$.
We need to distinguish two cases depending on the sign of $x_i - y_i$. 
If $x_i - y_i > 8 e \xi^{-1} C \ep$, we will show that the geodesic $P_{\bd y}^0$ in the definition of $E^\ep_{\bd y}(z,w)$ cannot exist (since there will be a shorter path with the same endpoints). If $x_i - y_i < - 8 e \xi^{-1} C \ep$, we will similarly show that the geodesic $P_{\bd y}^i$ cannot exist.
\medskip 

\noindent\textit{Case 1: $x_ i -y_i > 8 e \xi^{-1} C \ep$.}
Let $z^i \in  \mcl B_\ep(z ; D_h)$ and $w^i \in \mcl B_\ep(w ; D_h)$ be the endpoints of $P_{\BB x}^i$.  
Since the $D_{h_{\bd x}}$-geodesic $P_{\bd x}^i$ enters $U_i$ and $z^i,w^i\notin U_i$, there are times $ 0 < \tau < \sigma < D_{h_{\bd x}}(z^i,w^i)$ such that $P_{\bd x}^i([\tau,\sigma]) \subset \ol{U_i' \setminus U_i}$.
By the definition~\eqref{eqn-geo-count-reg} of $G (C)$ and since $\phi_i \geq 0$ for each $i\in [1,n]_{\BB Z}$ (which implies $D_{h_{\bd x}} \geq D_h$), 
\eqb \label{eqn-geo-count-length}
\tau - \sigma 
= \op{len}\left( P_{\bd x}^i|_{[\tau,\sigma]} ; D_{h_{\bd x}} \right)  
\geq \op{len}\left( P_{\bd x}^i|_{[\tau,\sigma]} ; D_h \right)  
\geq 1/C .
\eqe

By the definition of $E_{\bd x}^\ep(z,w)$, the geodesic $P_{\bd x}^i$ is disjoint from $V_j$ for $j\not= i$. 
Since $\phi_j$ for $j\not=i$ is supported on $V_j$, it follows from Weyl scaling (Axiom~\ref{item-metric-f}) that the $D_{h_{\bd y}}$-length of $P_{\bd x}^i$ is the same as the LQG length of $P_{\bd x}^i$ w.r.t.\ the field $h + y_i \phi_i = h_{\bd x} - (x_i - y_i) \phi_i$. 

Since $\phi_i \geq 0$, $x_i - y_i > 0$, and $\phi_i$ is identically equal to 1 on $U_i'$, Weyl scaling gives
\allb \label{eqn-geo-count-weyl}
&\op{len}\left( P_{\bd x}^i|_{[\tau,\sigma]} ; D_{h_{\bd y}} \right) = e^{-\xi(x_i-y_i)} (\sigma-\tau) 
\quad \text{and} \notag \\
&\qquad\op{len}\left( P_{\bd x}^i|_{[0,\tau] \cup [\sigma , D_{h_{\bd x}}(z^i,w^i)]} ; D_{h_{\bd y}} \right) \leq   D_{h_{\bd x}}(z^i,w^i) - (\sigma-\tau)    .
\alle
Using~\eqref{eqn-geo-count-weyl}, followed by~\eqref{eqn-geo-count-length} and the fact that $x_i - y_i >8 e \xi^{-1} C  \ep$, we therefore have
\allb \label{eqn-geo-count-neg}
\op{len}\left( P_{\bd x}^i ; D_{h_{\bd y}} \right) 
&\leq  D_{h_{\bd x}}(z^i,w^i) - (1 - e^{-\xi(x_i - y_i)}) (\sigma-\tau)  \notag\\
&<  D_{h_{\bd x}}(z^i,w^i) - (1 - e^{-8 C  e  \ep}) / C
\leq D_{h_{\bd x}}(z^i,w^i) - 8 \ep .
\alle
Note that in the last inequality, we used that $1-e^{-s} \geq s/e$ for $s\in [0,1]$. 
 
Now consider a path $\wt P$ from a point $\wt z \in \mcl B_\ep(z ; D_{h_{\bd y}})$ to a point $\wt w \in \mcl B_\ep(w ; D_{h_{\bd y}})$ which does not enter $\bigcup_{i=1}^{n} \ol V_i$.
We will use~\eqref{eqn-geo-count-neg} to construct another path from $\wt z$ to $\wt w$ which is strictly $D_{h_{\bd y}}$-shorter than $\wt P$, which will show that $\wt P$ cannot be a $D_{h_{\bd y}}$-geodesic. This will then imply that $E_{\bd y}(z,w)$ cannot occur (since the geodesic $P_{\bd y}^0$ in the definition of  $E_{\bd y}^\ep(z,w)$ cannot exist). 

Since $E^\ep_{\bd x}(z,w)$ occurs, Lemma~\ref{lem-geo-count-disjoint} implies that $\mcl B_\ep(z ;D_{h_{\bd y}}) =  \mcl B_\ep(z ;D_{h_{\bd x}})  $ and $\mcl B_\ep(w;D_{h_{\bd y}})  =\mcl B_\ep(w;D_{h_{\bd x}})$. In particular, each of these balls is disjoint from $\bigcup_{i=1}^n \ol V_i$. 
We can find a path $\wt P_{z^i}$ from $\wt z$ to $z^i$ which is contained in $\mcl B_\ep(z;D_{h_{\bd y}})$ and a path $\wt P_{w^i}$ from $w^i$ to $\wt w$ which is contained in $\mcl B_\ep(w;D_{h_{\bd y}})$ such that each of $\wt P_{z^i}$ and $\wt P_{w^i}$ have $D_{h_{\bd y}}$-length at most $2\ep$.
 
The concatenation $\wt P'$ of $\wt P_{z^i}$, $P_{\bd x}^i$, and $\wt P_{w^i}$ is a path from $z^i$ to $w^i$ which does not enter $\bigcup_{i=1}^n \ol V_i$ and which has $D_{h_{\bd y}}$-length at most $\op{len}(P_{\bd x}^i ; D_{h_{\bd y}}) + 4\ep$. By~\eqref{eqn-geo-count-neg},
\eqb \label{eqn-geo-count-shorter}
\op{len}\left( \wt P' ; D_{h_{\bd y}} \right) < D_{h_{\bd x}}(z^i,w^i)  -4 \ep . 
\eqe

On the other hand, $\wt P$ is disjoint from $\bigcup_{i=1}^n \ol V_i$ so the $D_{h_{\bd x}}$-length of $\wt P$ is the same as its $D_{h_{\bd y}}$-length. 
In particular,
\eqb  \label{eqn-geo-count-longer}
\op{len}\left( \wt P  ; D_{h_{\bd y}} \right) \geq D_{h_{\bd x}}(\wt z , \wt w) , 
\eqe
which by the triangle inequality is at least $D_{h_{\bd x}}(z^i,w^i) - 4\ep$. Combined with~\eqref{eqn-geo-count-shorter}, this shows that $\wt P'$ is strictly $D_{h_{\bd y}}$-shorter than $\wt P$, so $\wt P$ is not a $D_{h_{\bd y}}$-geodesic. 
\medskip 

\noindent\textit{Case 2: $x_ i -y_i < -8 e \xi^{-1} C \ep$.}
Now consider a path $\wt P$ from a point $\wt z \in \mcl B_\ep(z ; D_{h_{\bd y}})$ to a point $\wt w \in \mcl B_\ep(q ; D_{h_{\bd y}})$ which enters $U_i$. 
We will show that $\wt P$ cannot be a $D_{h_{\bd y}}$-geodesic, and hence that $E_{\bd y}^\ep(z,w)$ cannot occur (since the geodesic $P_{\bd y}^i$ in the definition of $E_{\bd y}^\ep(z,w)$ cannot exist). 
Via a similar calculation to the one leading to~\eqref{eqn-geo-count-neg}, we obtain
\eqb \label{eqn-geo-count-pos}
\op{len}\left( \wt P ; D_{h_{\bd y}} \right)  
> D_{h_{\bd x}}(\wt z , \wt w) + 8 \ep .
\eqe

We will now construct a new path from $\wt z$ to $\wt w$ whose $D_{h_{\bd y}}$-length is strictly shorter than that of $\wt P$. 
Recall that $P_{\bd x}^0$ is a $D_{h_{\bd x}}$-geodesic from a point of $\mcl B_\ep(z;D_{h_{\bd x}})$ to a point of $\mcl B_\ep(w; D_{h_{\bd y}})$ which is disjoint from $\bigcup_{i=1}^n \ol V_i$. 
Let $z^0\in \mcl B_\ep(z ; D_{h_{\bd x}})$ and $w^0 \in \mcl B_\ep(w ; D_{h_{\bd x}})$ be the endpoints of $P_{\bd x}^0$. 
Similarly as in Case 1, we concatenate $P_{\bd x}^0$ with a path in $\mcl B_\ep(z ; D_{h_{\bd x}})$ from $\wt z$ to $z^0$ with $D_{h_{\bd x}}$-length at most $2\ep$ and a path in $\mcl B_\ep(w ; D_{h_{\bd x}})$ from $w^0$ to $\wt w$ with $D_{h_{\bd x}}$-length at most $2\ep$.
This gives a path $\wt P'$ from $\wt z$ to $\wt w$ which is disjoint from $\bigcup_{i=1}^n \ol V_i$ and has $D_{h_{\bd x}}$-length at most $D_{h_{\bd x}}(z^0 ,w^0 ) + 4\ep$. 

Since each $\phi_i$ is supported on $V_i$, it follows from Weyl scaling that 
\eqb \label{eqn-geo-count0}
\op{len}\left( \wt P ' ; D_{h_{\bd y}} \right) 
= \op{len}\left( \wt P ' ; D_{h_{\bd x}} \right)
\leq D_{h_{\bd x}}(z^0 ,w^0 ) + 4\ep ,
\eqe
which by the triangle inequality is at most $D_{h_{\bd x}}(\wt z , \wt w) + 8\ep$. 
Comparing this with~\eqref{eqn-geo-count-pos} shows that
\eqbn
\op{len}\left( \wt P ' ; D_{h_{\bd y}} \right)  <  \op{len}\left( \wt P ; D_{h_{\bd y}} \right).
\eqen
Hence $\wt P$ cannot be a $D_{h_{\bd y}}$-geodesic, as required.
\end{proof}

\begin{lem} \label{lem-geo-count-unif}
Let $\bd X = (X_1,\dots,X_n)$ be sampled uniformly from Lebesgue measure on $[0,1]^n$, independently from $h$. 
For each $C>1$, each $z,w\in\BB C$, and each $\ep > 0$, a.s.\
\eqb \label{eqn-geo-count-unif}
\BB P\left[ E_{\bd X}^\ep(z,w) \cap G_{\bd X}(C) \,|\, h \right] = O_\ep(\ep^n) 
\eqe
at a rate which is deterministic and uniform over all $z,w\in\BB C$. 
\end{lem}
\begin{proof}
By Lemma~\ref{lem-geo-count-perturb}, if $\ep \in (0,1/(100 C))$ and $G(C)$ occurs then the Lebesgue measure of the set of $\bd x \in [0,1]^n$ for which $E_{\bd x}^\ep(z,w)$ occurs is at most $(16 e \xi^{-1} C  \ep)^n$. Consequently,
\eqb \label{eqn-geo-count-unif0} 
\BB P\left[ E_{\bd X}^\ep(z,w) \cap G (C) \,|\, h \right] \leq (16 e \xi^{-1} C  \ep)^n .
\eqe
Recall that the random function $\sum_{i=1}^n X_i \phi_i$ takes values in $[0,1]$. 
By the Weyl scaling property of $D_h$, if $G_{\bd X}(C)$ occurs, then $G(e^\xi C)$ occurs. Therefore,~\eqref{eqn-geo-count-unif} for $\ep \in (0,1/(100 C))$ follows from~\eqref{eqn-geo-count-unif0}. This implies~\eqref{eqn-geo-count-unif} in general since the implicit constant in the $O_\ep(\cdot)$ is allowed to depend on $C$. 
\end{proof}

\begin{proof}[Proof of Lemma~\ref{lem-geo-count-prob}]
Let $\bd X = (X_1,\dots,X_n)$ be a vector of i.i.d.\ uniform random variables as in Lemma~\ref{lem-geo-count-unif}.
By~\eqref{eqn-multi-geo-bdy}, the support of the function $\sum_{i=1}^n X_i\phi_i$ is disjoint from $\bdy\BB D$, so the circle average of $h_{\bd x}$ over $\bdy\BB D$ is zero.
By a standard Radon-Nikodym derivative calculation for the GFF (see, e.g.,~\cite[Lemma A.2]{gwynne-ball-bdy}), if we condition on $\bd X$ then the laws of $h$ and $h_{\bd X}$ are mutually absolutely continuous and the Radon-Nikodym derivative of the law of $h$ w.r.t.\ the conditional law of $h_{\bd X}$ is given by
\eqb
M = M(h_{\bd X} , \bd X) = \exp\left( - \sum_{i=1}^n \left( X_i (h_{\bd X} , \phi_i)_\nabla - \frac{X_i^2}{2} (\phi_i , \phi_i)_\nabla \right) \right) .
\eqe
We have
\eqbn
(h_{\bd X} , \phi_i)_\nabla = (h  , \phi_i)_\nabla + X_i (\phi_i , \phi_i)_\nabla  .
\eqen
Since $(h,\phi_i)_\nabla$ is centered Gaussian with variance $(\phi_i,\phi_i)_\nabla$, we can compute that for each $p>1$,  
\eqb
\BB E\left[ M^p \,|\, \bd X \right] \leq \exp\left(   \frac{p^2 - p}{2}  \sum_{i=1}^n X_i^2  (\phi_i , \phi_i)_\nabla \right)
\eqe
Each $X_i$ takes values in $[0,1]$ and there is some deterministic constant $A$ (which does not depend on $z,w,\ep$) such that $(\phi_i,\phi_i)_\nabla  \leq A $ for each $i\in [1,n ]_{\BB Z}$. Therefore, 
\eqb \label{eqn-geo-count-rn}
\BB E\left[ M^p \,|\, \bd X \right]  = O_\ep(1) .
\eqe  

We now use H\"older's inequality to get that if $p,q > 1$ with $1/p + 1/q = 1$, then
\allb
\BB P\left[ E^\ep(z,w) \cap G(C) \right] 
&= \BB P\left[ E^\ep(z,w) \cap G(C)  \,|\, \bd X \right] \quad \text{($\bd X$ and $h$ are independent)} \notag\\
&= \BB E\left[ M \BB 1_{E^\ep_{\bd X}(z,w) \cap G_{\bd X}(C)} \,|\, \bd X \right] \notag\\
&\leq \BB E\left[ M^p   \,|\, \bd X \right]^{1/p}    \BB P\left[   E^\ep_{\bd X}(z,w) \cap G_{\bd X}(C)   \,|\, \bd X \right]^{1/q} \quad \text{(H\"older)}  .
\alle
We now take unconditional expectations of both sides of this last display, and use~\eqref{eqn-geo-count-rn} and Lemma~\ref{lem-geo-count-unif} to bound the right side. This leads to
\eqb
\BB P\left[ E^\ep(z,w) \cap G(C) \right] 
\preceq  \BB E\left[ \BB P\left[   E^\ep_{\bd X}(z,w) \cap G_{\bd X}(C)   \,|\, \bd X \right]^{1/q}  \right]
\leq   \BB P\left[   E^\ep_{\bd X}(z,w) \cap G_{\bd X}(C)   \right]^{1/q}
= O_\ep(\ep^{n/q} ). 
\eqe
Note that in the second inequality, we used Jensen's inequality to move the $1/q$ outside of the expectation. 
Sending $q \rta 1^+$ (equivalently, $p \rta \infty$) yields~\eqref{eqn-geo-count-prob}. 
\end{proof}

\subsection{Proof of Proposition~\ref{prop-subset-dim}} 
\label{sec-subset-dim}

Since we are working with LQG balls rather than Euclidean balls, we will need to apply Lemma~\ref{lem-geo-count-prob} for points $z,w$ sampled from the $\gamma$-LQG measure $\mu_h$, rather than for deterministic points. We can do this due to the description of the conditional law of $h$ given points sampled from its LQG area measure (Lemma~\ref{lem-two-point-law}) and the following extension of Lemma~\ref{lem-geo-count-prob}. 

\begin{lem} \label{lem-geo-count-prob-f}  
Fix a constant $A > 1$. Let $f : \BB C\rta \BB R \cup \{\infty\}$ be a deterministic function which is continuous except for finitely many logarithmic singularities. Assume that $D_{h+f}$ a.s.\ induces the same topology on $\BB C$ as the Euclidean metric and 
\eqb
\sup\left\{ |f(u)| : u \in \bigcup_{i=1}^{n} \ol V_i \right\} \leq A . 
\eqe
For $z,\in\BB C$ and $\ep > 0$, define $E^\ep_f(z,w)$ in the same manner as the event $E^\ep(z,w)$ from Section~\ref{sec-geo-count-prob} but with $h+f$ in place of $f$. 
For each $C>1$,  
\eqb \label{eqn-geo-count-prob-f}
\BB P\left[ E_f^\ep(z,w) \cap G(C) \right] \leq \ep^{n + o_\ep(1)}
\eqe
where here the rate of the $o_\ep(1)$ does not depend on $z$ or $w$ and depends on $f$ only via the constant $A$. 
\end{lem}
\begin{proof}
This follows from exactly the same proof as Lemma~\ref{lem-geo-count-prob}, with $h+f$ used in place of $h$ throughout. 
To be more precise, the proofs of Lemmas~\ref{lem-geo-count-disjoint}, \ref{lem-geo-count-perturb}, and~\ref{lem-geo-count-unif} do not use any particular properties of the field $h$ so carry over verbatim to our setting, except that the constant $8e\xi^{-1} C$ appearing in Lemma~\ref{lem-geo-count-perturb} is replaced by a possibly larger constant which is allowed to depend on $A$.
The only property of $h$ used in Lemma~\ref{lem-geo-count-prob} is the Radon-Nikodym derivative bound, but this is identical in our setting since $f$ is deterministic so for any deterministic smooth function $g$, the Radon-Nikodym derivative between the laws of $h+f+g$ and $h+g$ is the same as the Radon-Nikodym derivative between the laws of $h$ and $h+g$. 
\end{proof}

We now prove a variant of Lemma~\ref{lem-geo-count-prob} for a random choice of $z$ and $w$.

\begin{lem} \label{lem-random-pt}
Let $Z,W\subset\BB C$ be bounded open sets which lie at positive distance from each other and from $ \bigcup_{i=1}^n \ol V_i$. 
Conditional on $h$, let $(\BB z,\BB w)$ be sampled from $\mu_h|_Z \times \mu_h|_W$, normalized to be a probability measure.
For each $C>1$ and each $\ep\in (0,1)$, 
\eqb \label{eqn-random-pt}
\BB P\left[ E^\ep(\BB z,\BB w) \cap G(C)  \right] \leq \ep^{n + o_\ep(1)} 
\eqe
where the rate of the $o_\ep(1)$ is deterministic.
\end{lem}
\begin{proof}
Write $\wt{\BB P}$ for the law of $(h,\BB z,\BB w)$ weighted by $\mu_h(Z)\mu_h(W)$, normalized to be a probability measure. 
Write $\wt{\BB E}$ for the corresponding expectation.  
By the description~\eqref{item-law-quantum} from Lemma~\ref{lem-two-point-law}, under $\wt{\BB P}$ the conditional law of $h$ given $(\BB z,\BB w)$ is the same as the the law of the field $\wt h - \gamma \log|\cdot-\BB z| - \gamma \log |\cdot-\BB w|$, normalized to have average zero over $\bdy\BB D$, where $\wt h$ is a whole-plane GFF which is independent from $\BB z , \BB w$. 
Hence we can apply Lemma~\ref{lem-geo-count-prob} under the conditional law given $\BB z , \BB w$, with $f$ given by $ -\gamma \log|\cdot-\BB z| - \gamma \log |\cdot-\BB w|$ minus its average over $\bdy\BB D$ to get that
\eqb \label{eqn-random-pt-weighted}
\wt{\BB P}\left[ E^\ep(\BB z,\BB w) \cap G(C)  \right]
= \wt{\BB E}\left[ \wt{\BB P}\left[ E^\ep(\BB z,\BB w) \cap G(C) \,|\, \BB z,\BB w \right] \right] 
= \ep^{n+o_\ep(1)} .
\eqe 

We now use H\"older's inequality to get that for $p , q > 1$ with $1/p+1/q=1$, 
\allb \label{eqn-random-pt-holder} 
\BB P\left[ E^\ep(\BB z,\BB w) \cap G(C)   \right] 
&= c^{-1} \wt{\BB E}\left[ \BB 1_{ E^\ep(\BB z,\BB w) \cap G(C)}  \mu_h(Z)^{-1} \mu_h(W)^{-1}  \right] \notag\\
&= c^{-1} \wt{\BB P}\left[  E^\ep(\BB z,\BB w) \cap G(C) \right]^{1/p} \times \wt{\BB E}\left[ \mu_h(Z)^{-q} \mu_h(W)^{-q} \right]^{1/q} ,
\alle
where $c = \wt{\BB E}[ \mu_h(Z)^{-1} \mu_h(W)^{-1}  ]$ is a normalizing constant.
By~\eqref{eqn-random-pt-weighted}, the first factor on the right side of~\eqref{eqn-random-pt-holder} is at most $\ep^{n/p + o_\ep(1)}$. 
The second factor on the right side of~\eqref{eqn-random-pt-holder} is equal to $\BB E\left[ \mu_h(Z)^{1-q} \mu_h(W)^{1-q} \right]^{1/q}$, which is finite for any choice of $q$ since $\mu_h$ has negative moments of all orders~\cite[Theorem 2.12]{rhodes-vargas-review}. 
Sending $p\rta 1$ and $q\rta \infty$ now gives~\eqref{eqn-random-pt}.
\end{proof}

Recall the set $\wt{\mcl S}_n$ from Proposition~\ref{prop-subset-dim}. 

\begin{lem} \label{lem-subset-cover}
Fix bounded open sets $Z,W \subset \BB C$ which lie at positive distance from each other and from $ \bigcup_{i=1}^n \ol V_i$. 
Also fix compact sets $K_1\subset Z$ and $K_2\subset W$.  
With probability tending to 1 as $\ep\rta 0$, the set $\wt{\mcl S}_n \cap (K_1\times K_2)$ can be covered by $\ep^{n-2d_\gamma+o_\ep(1)}$ sets of the form $\mcl B_\ep(z;D_h) \times \mcl B_\ep(w;D_h)$ for $(z,w) \in Z\times W$. 
\end{lem}
\begin{proof}
Let $\zeta\in (0,1)$ be a small constant (which we will eventually send to zero) and let $M_\ep := \lfloor \ep^{-d_\gamma - \zeta}\rfloor$. Also let $C>1$ (which we will eventually send to $\infty$). 
Conditional on $h$, let $\{\BB z_k\}_{k\in [1,M_\ep]_{\BB Z}}$ (resp.\ $\{\BB w_k\}_{k\in [1,M_\ep]_{\BB Z}} $) be conditionally i.i.d.\ samples from $\mu_h|_{Z}$ (resp.\ $\mu_h|_{W }$), normalized to be a probability measure.
 
By Lemma~\ref{lem-ball-cover}, it holds with probability tending to 1 as $\ep\rta 0$ that the following is true.
\begin{itemize}
\item The sets $\mcl B_\ep(\BB z_k ; D_h) \times \mcl B_\ep(\BB w_l ; D_h)$ for $(k,l) \in [1,M^\ep]_{\BB Z}^2$ cover $K_1\times K_2$.
\item For each $(k,l) \in [1,M^\ep]_{\BB Z}^2$, the sets $\mcl B_\ep(\BB z_k ; D_h)$, $\mcl B_\ep(\BB w_l ; D_h)$, and $\bigcup_{i=1}^n \ol V_i$ are disjoint. 
\end{itemize}
We claim that if the above two conditions hold, then $\wt{\mcl S}_n \cap (K_1\times K_2)$ is contained in the union of the sets $\mcl B_\ep(\BB z_k ; D_h) \times \mcl B_\ep(\BB w_l ; D_h)$ for $(k,l) \in [1,M^\ep]_{\BB Z}^2$ such that $E^\ep(\BB z_k , \BB w_l)$ occurs.
Indeed, if the above two conditions hold and $(\wt z,\wt w) \in \wt{\mcl S}_n \cap (K_1\times K_2)$ then there exists $(k,l) \in [1,M^\ep]_{\BB Z}^2$ such that $(\wt z,\wt w) \in \mcl B_\ep(\BB z_k ; D_h) \times \mcl B_\ep(\BB w_l ; D_h)$. 
It is easy to see from the definitions of $\wt{\mcl S}_n$ and $E^\ep(\BB z_k , \BB w_l)$ that $E^\ep(\BB z_k , \BB w_l)$ occurs for this choice of $k,l$ (see~\eqref{eqn-subset-def} and~\eqref{eqn-subset-def'}). 
 
By Lemma~\ref{lem-random-pt} and the Chebyshev inequality, it holds with probability tending to 1 as $\ep\rta 0$ that the number of pairs $(\BB z_k , \BB w_l)$ for $(k,l) \in [1,M_\ep]_{\BB Z}^2$ for which $E^\ep(\BB z_k , \BB w_l) \cap G(C)$ occurs is at most $\ep^{ - 2d_\gamma - 2\zeta + n+o_\ep(1)}$.  
Since $\BB P[G(C)] \rta 1$ as $C\rta \infty$, the lemma statement follows by sending $C\rta \infty$ and $\zeta \rta 0$ at a sufficiently slow rate as $\ep\rta 0$.
\end{proof}

\begin{proof}[Proof of Proposition~\ref{prop-subset-dim}]
Let $Z,W \subset \BB C$ be bounded open sets which lie at positive distance from each other, from $ \bigcup_{i=1}^n V_i$. 
By Lemma~\ref{lem-subset-cover}, if $K_1\subset W_1$ and $K_2\subset W_2$ are as in that lemma, then for $n\leq \lfloor 2d_\gamma  \rfloor$ a.s.\ $\dim_{\mcl H}^\gamma(\wt{\mcl S}_n \cap (K_1\times K_2)) \leq 2d_\gamma-n$ and for $n > \lfloor 2d_\gamma   \rfloor$ a.s.\ $ \wt{\mcl S}_n \cap (K_1\times K_2)  = \emptyset$. 
Letting $K_1$ and $K_2$ increase to $Z$ and $W$, respectively, gives the same statement with $Z\times W$ in place of $K_1\times K_2$. 

By the definition of $\wt{\mcl S}_n$, if $(z,w) \in \wt{\mcl S}_n$ then $z\not=w$ and $z$ and $w$ each lie at positive distance from $\bigcup_{i=1}^n \ol V_i$. 
Hence there are Euclidean balls $B_1,B_2$ with rational centers and radii such that $z \in B_1$, $w \in B_2$, and $B_1$ and $B_2$ each lie at positive distance from $\bigcup_{i=1}^n \ol V_i$. 
By the preceding paragraph applied with $Z = B_1$ and $W = B_2 $, together with the countable stability of Hausdorff dimension, we now obtain the proposition statement. 
\end{proof}

\subsection{Upper bound on the number of distinct geodesics} 
\label{sec-geo-finite}

In this section we will deduce Theorem~\ref{thm-finite-geo} from Proposition~\ref{prop-multi-geo-dim}. The idea is that if $\mcl P$ is a finite collection of geodesics with the same endpoints, then we can always find a subset $\wt{\mcl P}$ of $\mcl P$ with the following properties. The cardinality $\#\wt{\mcl P}$ is bounded below in terms of $\#\mcl P$ and the geodesics in $\wt{\mcl P}$ satisfy the non-overlapping property in the definition of $\mcl S_n$ for $n = \#\wt{\mcl P} -1$. 
We will then use the fact that $\mcl S_n = \emptyset$ for $n > \lfloor 2d_\gamma \rfloor$ to get an upper bound for $\#\mcl P$. 

The construction of the desired subset $ \wt{\mcl P}$ of $\mcl P$ uses a combination of known topological properties of geodesics started from rational points and purely combinatorial arguments. The key observation for the topological part of the argument is the following dichotomy.

\begin{figure}[t!]
 \begin{center}
\includegraphics[scale=.9]{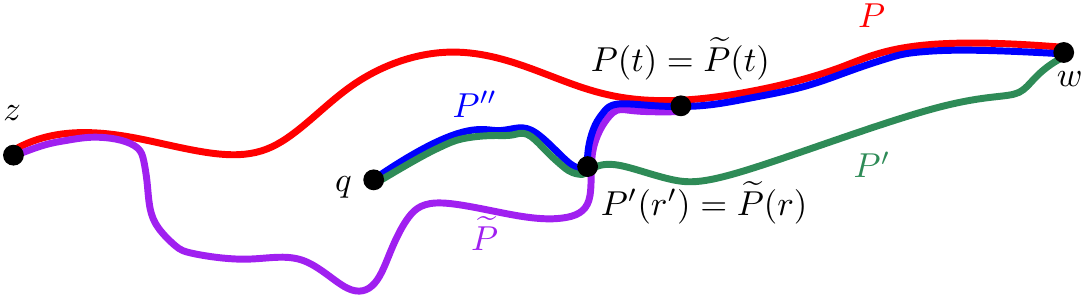}
\vspace{-0.01\textheight}
\caption{Illustration of the proof of Lemma~\ref{lem-dichotomy}. If there are two distinct geodesics from $z$ to $P(t)$, say $P|_{[0,t]}$ and $\wt P$, then we can use Lemma~\ref{lem-regular-geo'} to find a geodesic from $q\in\BB Q^2$ to $w$ which merges into $P$ before reaching $P(t)$. Condition~\ref{item-dichotomy-three} at time $t$ then follows from the analogous properties of geodesics started from rational points (Lemma~\ref{lem-regular-geo'} and Proposition~\ref{prop-four-geo}). 
}\label{fig-dichotomy}
\end{center}
\vspace{-1em}
\end{figure}

\begin{lem} \label{lem-dichotomy}
Almost surely, the following is true.
For every distinct $z,w\in\BB C$, every $D_h$-geodesic $P$ from $z$ to $w$, and every $t\in (0,D_h(z,w))$, at least one of the following two conditions holds.
\begin{enumerate}[A.]
\item $P|_{[0,t]}$ is the unique geodesic from $z$ to $P(t)$. \label{item-dichotomy-unique}
\item For every $s\in (t,D_h(z,w))$, $P|_{[t,s]}$ is the unique geodesic from $P(t)$ to $P(s)$. Moreover, there are at most three geodesics from $P(t)$ to $w$.  \label{item-dichotomy-three}
\end{enumerate}
Furthermore, the set of  $t\in (0,D_h(z,w))$ for which condition~\ref{item-dichotomy-unique} (resp.\ \ref{item-dichotomy-three}) holds is a sub-interval of $(0,D_h(z,w))$ with $0$ (resp.\ $D_h(z,w)$) as one of its endpoints (we do not specify whether the endpoint of the interval is included). 
\end{lem}
\begin{proof}
See Figure~\ref{fig-dichotomy} for an illustration of the proof. 
The proof is an easy consequence of the following two facts. 
\begin{enumerate}
\item By Lemma~\ref{lem-regular-geo'}, a.s.\ for every $q\in\BB Q^2$, every $w\in\BB C$, every geodesic $P$ from $q$ to $w$, and every $s \in [0,D_h(q,w))$, $P|_{[0,s]}$ is the unique geodesic from $q$ to $P(s)$. \label{item-dichotomy-reg}
\item By Proposition~\ref{prop-four-geo}, a.s.\ for every $q\in\BB Q^2$ and every $w\in\BB C$, there are at most three geodesics from $q$ to $w$. \label{item-dichotomy-four}
\end{enumerate}
Henceforth assume that both of the above numbered properties hold, which happens with probability one. 

Let $z,w,P$, and $t$ be as in the lemma statement. Assume that there are at least two distinct geodesics from $z$ to $P(t)$. 
We will show that condition~\ref{item-dichotomy-three} holds. 
Let $\wt P$ be a geodesic from $z$ to $P(t)$ which is not equal to $P|_{[0,t]}$. 
The set $\BB C\setminus (P([0,t]) \cup \wt P([0,t]))$ lies at positive distance from $w$ and has at least two connected components. Since a geodesic is a simple curve, $P|_{(t,D_h(z,w)]}$ is entirely contained in one of these connected components. Let $U$ be a connected component other than the one which contains $P|_{(t,D_h(z,w)]}$. 

Let $q \in \BB Q^2 \cap U$ and let $P'$ be a geodesic from $q$ to $w$. 
Since $w \notin U$, the geodesic $P'$ must cross either $P([0,t])$ or $\wt P([0,t])$. That is, there must be times $r \in [0,t]$ and $r' \in [0,D_h(q,w)]$ for which either $P(r) = P'(r')$ or $\wt P(r) = P'(r')$. Assume that $\wt P(r) = P'(r')$ (the other case is almost the same, but slightly simpler). 

The concatenation of $\wt P|_{[r,t]}$ and $P|_{[t,D_h(z,w)]}$ is a geodesic from $\wt P(r)$ to $w$. Since $P'(r') = \wt P(r)$, it follows that the concatenation $P''$ of $P'|_{[0,r']}$, $\wt P|_{[r,t]}$, and $P|_{[t,D_h(z,w)]}$ is a geodesic from $q$ to $w$. By property~\ref{item-dichotomy-reg} above, applied to the geodesic $P''$, for each $s\in (t,D_h(z,w))$, there is a unique geodesic from $q$ to $P(s)$, namely the segment of $P''$ before it hits $P(s)$. Therefore, the segment of $P''$ between $P(t)$ and $P(s)$, namely $P|_{[t,s]}$, is the unique geodesic from $P(t)$ to $P(s)$. 

By property~\ref{item-dichotomy-four}, there are at most three geodesics from $q$ to $w$. If there were more than three geodesics between $P(t)$ and $w$, then we could get more than three geodesics from $q$ to $w$ by replacing the segment of $P''$ between $P(t)$ and $w$ by one of the other geodesics from $P(t)$ to $w$. 
Hence there must be at most three geodesics from $P(t)$ to $w$. 
This gives the desired dichotomy.

The last statement, about the form of the sets for which each of the two conditions hold, is an easy consequence of the following observation: if $0 \leq t < t' < s \leq D_h(z,w)$, then the number of $D_h$-geodesics from $P(t)$ to $P(s)$ is at least the number of $D_h$-geodesics from $P(t')$ to $P(s)$.  
\end{proof}

Lemma~\ref{lem-dichotomy} has the following useful consequence, which says that any two $D_h$-geodesics with the same endpoints can make at most two ``excursions" away from each other.

\begin{lem} \label{lem-two-excursion}
Almost surely, the following is true.
Let $z,w\in\BB C$ and let $P,P'$ be $D_h$-geodesics from $z$ to $w$.
The set $\{t\in [0,D_h(z,w)] : P(t) \not= P'(t)\}$ has at most two connected components.
\end{lem}
\begin{proof}
Write $T =  \{t\in [0,D_h(z,w)] : P(t) \not= P'(t)\}$. 
Let $(a,b)$ be a non-trivial connected component of $T$. Then $P(b)= P'(b)$ and there are at least two distinct geodesics from $z$ to $P(b)$, namely $P|_{[0,b]}$ and the concatenation of $P|_{[0,a]}$ and $P'|_{[a,b]}$. By Lemma~\ref{lem-dichotomy}, there are at most 3 distinct geodesics from $P(b)$ to $w$. On the other hand, the number of distinct geodesics from $P(b)$ to $w$ is at least $2^n$, where $n$ is the number of connected components of $[b,D_h(z,w)] \cap T$: indeed, this is because each such connected component gives rise to two possible paths with the same endpoints which a geodesic from $P(b)$ to $w$ could take. Since $2^n \leq 3$ we must have $n\leq 1$. Hence there is at most one connected component of $T$ lying strictly to the right of any given connected component of $T$. Therefore, $T$ has at most two connected components.
\end{proof}

The following lemma shows that we can find, for each geodesic in $\mcl P$, a point which is hit by at most two other geodesics in $\mcl P$. This should be compared to the definition of $\mcl S_n$, where each geodesic is required to have a point which is hit by \emph{no} other geodesic in the collection.

\begin{figure}[t!]
 \begin{center}
\includegraphics[scale=.9]{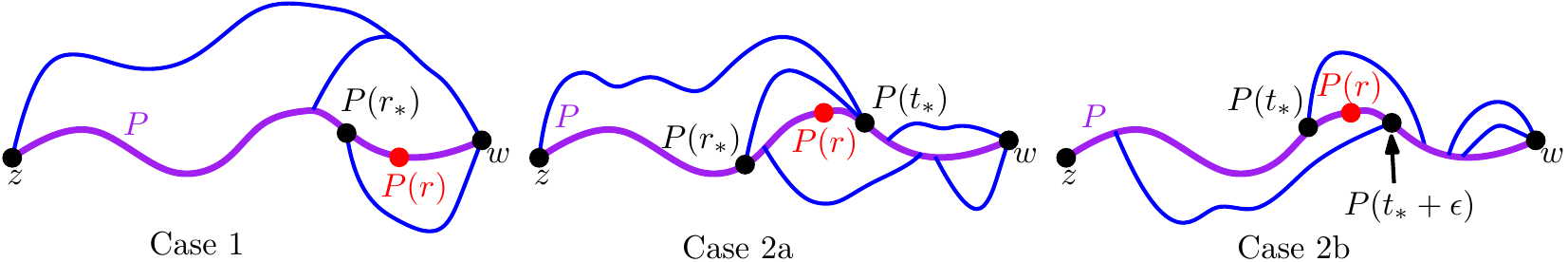}
\vspace{-0.01\textheight}
\caption{Illustration of each of the different cases treated in the proof of Lemma~\ref{lem-three-hit-pt}. In each case, the geodesic $P$ is shown in purple and segments of other geodesics in $\mcl P$ are shown in blue. Geodesics from $z$ to $w$ which are not in $\mcl P$ are not shown (although such geodesics must exist in Case 2b). 
}\label{fig-three-hit}
\end{center}
\vspace{-1em}
\end{figure}

\begin{lem}  \label{lem-three-hit-pt}
Almost surely, the following is true. 
Let $z,w\in\BB C$ be distinct and let $\mcl P$ be a finite collection of distinct geodesics from $z$ to $w$. 
For each $P\in\mcl P$, there is a time $r = r(P) \in (0,D_h(z,w))$ such that $P(r)$ is hit by at most three geodesics in $\mcl P$ (including $P$ itself). 
\end{lem}
\begin{proof}
See Figure~\ref{fig-three-hit} for an illustration of the proof. 
Throughout the proof, we work on the full probability event of Lemma~\ref{lem-dichotomy}. 
Let $P\in \mcl P$. We consider two cases. 
\medskip

\noindent\textit{Case 1: For each $t\in (0,D_h(z,w))$, $P|_{[0,t]}$ is the unique geodesic from $z$ to $P(t)$.} 
Suppose $P' \in \mcl P$, $P'\not=P$. We claim that the set of $t\in [0,D_h(z,w)]$ for which $P(t) \not= P'(t)$ is of the form $(t_{P'} , D_h(z,w))$ for some $t_{P'} \in (0,D_h(z,w))$. 
Indeed, if not, then there are times $0 < t_1 < t_2  < D_h(z,w)$ such that $P(t_1) \not= P'(t_1)$ and $P(t_2) = P'(t_2)$. 
This gives two distinct geodesics from  $z$ to $P(t_2)$, namely $P|_{[0,t_2]}$ and $P'|_{[0,t_2]}$, which contradicts our assumption.  
We must have $t_{P'} < D_h(z,w)$ since $P  \not= P' $. 

If we let
\eqbn
r_* := \max_{P' \not= P} t_{P'}
\eqen
then $r_* < D_h(z,w)$ since $\mcl P$ is finite. For each $r \in (r_* , D_h(z,w))$, the point $P(r)$ is not hit by any geodesic in $\mcl P$ other than $P$. 
\medskip

\noindent\textit{Case 2: there is a $t\in (0,D_h(z,w))$ for which $P|_{[0,t]}$ is not the unique geodesic from $z$ to $P(t)$.}
Let $t_*$ be the supremum of the set of $t\in (0,D_h(z,w))$ for which $P|_{[0,t]}$ is the unique geodesic from $z$ to $P(t)$. 
Then the set of $t\in (0,D_h(z,w))$ for which $P|_{[0,t]}$ is the unique geodesic from $z$ to $P(t)$ is either $(0,t_*)$ or $(0,t_*]$. 
Hence, our assumption implies that $t_* < D_h(z,w)$. 
We consider two sub-cases.
\medskip

\noindent\textit{Case 2a: the $D_h$-geodesic from $z$ to $P(t_*)$ is not unique.} By definition, $P|_{[0,t]}$ is the unique geodesic from $z$ to $P(t)$ for each $t\in (0,t_*)$. 
Using Case 1 with $P(t_*)$ in place of $w$ and $\{P'|_{[0,t_*]} : P' \in\mcl P, P'(t_*) = P(t_*)\}$ in place of $\mcl P$, we find that there is an $r_* \in (0,t_*)$ such that each geodesic in $\mcl P$ which hits both $P((r_*,t_*))$ and $P(t_*)$ must coincide with $P$ on $[0,t_*]$. 

By Lemma~\ref{lem-dichotomy} and the the fact that the geodesic from $z$ to $P(t_*)$ is not unique, there are at most three geodesics from $P(t_*)$ to $w$. So, if $P' \in \mcl P$ hits both $P((r_*,t_*))$ and $P(t_*)$, then there is only one possibility for $P'|_{[0,t_*]}$ and at most three possibilities for $P'|_{[t_* , D_h(z,w)]}$. 
Hence there are at most three geodesics in $\mcl P$ which hit both $P((r_*,t_*))$ and $P(t_*)$.  

Since the range of a geodesic is a closed set, if $P' \in \mcl P$ and $P'$ does not hit $P(t_*)$, then there is a number $r_{P'} \in (0,t_*)$ such that $P'$ does not intersect $P((r_{P'} , t_*))$. By combining this with the preceding paragraph, we get that if  
\eqb
r  \in \left( r_* \vee \max\{r_{P'}  : P' \in \mcl P\setminus \{P\}\} , t_* \right) 
\eqe
then there are at most three geodesics in $\mcl P$ which hit $P(r)$. 
\medskip

\noindent\textit{Case 2b: the $D_h$-geodesic from $z$ to $P(t_*)$ is unique.} 
This case can arise if there is a sequence of times $t_n > t_*$ with $t_n \rta t_*$ such that the geodesic from $z$ to $P(t_n)$ is not unique for each $n$. This would imply that there are infinitely many geodesics from $z$ to $w$. Theorem~\ref{thm-finite-geo} implies \emph{a posteriori} that this cannot happen, but we have not yet proven this theorem, so we still have to deal with this possibility. Note that $\mcl P$ is a finite set by definition, so in this case some of the geodesics from $z$ to $w$ would have to not be in $\mcl P$. 

For $P'\in \mcl P \setminus \{P\}$, let $S_{P'}$ be the set of endpoints of connected components of $\{t \in [0,D_h(z,w)] : P(t) \not= P'(t)\}$. 
By Lemma~\ref{lem-two-excursion}, $\# S_{P'} \leq 4$. Since $\mcl P$ is a finite set, there exists $\ep > 0$ such that $S_{P'} \cap (t_* , t_*+\ep) =\emptyset$ for every $P' \in \mcl P \setminus \{P\}$ (note that we do not rule out the possibility that $t_* \in S_{P'}$). If $P'\in\mcl P$ such that $P'$ intersects $P((t_* ,t_*+\ep))$, but $P'|_{[t_*,t_*+\ep]} \not= P|_{[t_*,t_*+\ep]}$, then there must be an element of $S_{P'}$ in $(t_* , t_*+\ep)$. 
It follows that for each $P' \in \mcl P$, either $P'$ is disjoint from $P((t_* ,t_*+\ep))$ or $P'|_{[t_*,t_*+\ep]} = P|_{[t_*,t_*+\ep]}$. 

Let $r\in (t_* , t_*+\ep)$. Then if $P' \in\mcl P \setminus \{P\}$ and $P(r) = P'(r)$, the preceding paragraph implies that $P'|_{[t_*,t_*+\ep]} = P|_{[t_*,t_*+\ep]}$. Since the $D_h$-geodesic from $z$ to $P(t_*)$ is unique, we must have $P'|_{[0,t_*]} = P|_{[0,t_*]}$ and hence $P'|_{[0,t_*+\ep]} = P|_{[0,t_*+\ep]}$. 
By the definition of $t_*$, the geodesic from $z$ to $P(t_*+\ep)$ is not unique, so by Lemma~\ref{lem-dichotomy} there are at most three geodesics from $P(t_*+\ep)$ to $w$. Hence there are at most three possibilities for $P'|_{[t_*+\ep , D_h(z,w)]}$. Therefore, there are at most three geodesics in $\mcl P$ which hit $P(r)$.  
\end{proof}

We now discuss the main combinatorial input needed for the proof of Theorem~\ref{thm-finite-geo}. 
For a graph $G$, an \emph{independent set} of $G$ is a set of vertices of $G$ no two of which are joined by an edge. 
The following result is~\cite[Theorem 1]{hs-independence-number}. 

\begin{lem}[\cite{hs-independence-number}] \label{lem-independent-set}
Let $G$ be a connected graph with $V$ vertices and $E$ edges. 
There is an independent set of $G$ of cardinality at least
\eqbn
\left\lfloor \frac12 \left( \left(2E + V + 1 \right) - \sqrt{\left( 2E + V + 1 \right)^2 - 4V^2} \right) \right\rfloor 
\eqen
\end{lem}

\begin{proof}[Proof of Theorem~\ref{thm-finite-geo}]
We will find  a deterministic $m\in\BB N$ such that the following is true a.s. 
Let $M\in\BB N$, let $z,w\in\BB C$, and let $\mcl P$ be a finite collection of $\#\mcl P = M$ geodesics from $z$ to $w$. 
Then $M  \leq m$. 

To this end, we first use Lemma~\ref{lem-three-hit-pt} to get that for each $P\in\mcl P$, there is a point $u_P \in P$ such that $u_P$ is hit by at most three geodesics in $\mcl P$ (including $P$ itself). Just below, we will show using Lemma~\ref{lem-independent-set} that there is a subset $\wt{\mcl P}$ of $\mcl P$ such that 
\eqb \label{eqn-independent-set}
\#\wt{\mcl P} \geq   \left\lfloor \frac12 \left(\left( \frac{17}{3} M  + 1 \right) - \sqrt{\left( \frac{17}{3} M  + 1 \right)^2 - 4M^2} \right) \right\rfloor 
\eqe
and for each $P\in\wt{\mcl P}$, $P$ is the only geodesic in $\wt{\mcl P}$ which hits $u_P$. 
By Proposition~\ref{prop-multi-geo-dim}, $\#\wt{\mcl P} \leq \lfloor 2d_\gamma \rfloor + 1$. Since the right side of~\eqref{eqn-independent-set} goes to $\infty$ as $M\rta\infty$, we get that $M$ is bounded above by some finite, deterministic constant $m$.

To construct the set $\wt{\mcl P}$ of~\eqref{eqn-independent-set}, we first let $G_0$ be the graph whose vertex set is $\mcl P$ with distinct geodesics $P,P'\in\mcl P$ joined by an edge if and only if either $P$ hits $u_{P'}$ or $P'$ hits $u_P$. 
Then $G_0$ has $M$ vertices.
Since each $u_P$ is hit by at most 2 geodesics in $\mcl P$ other than $P$ itself, it follows that $G_0$ has at most $2M$ edges. 

Lemma~\ref{lem-independent-set} requires that the graph $G$ is connected. To arrange this in our setting, we first let $G_1$ be obtained from $G_0$ as follows.
We can assume without loss of generality that $M\geq 3$. 
For each $P\in\mcl P$ such that $P$ is joined by edges to fewer than 2 other vertices of $\mcl P$, we add one or two edges going from $P$ to (arbitrary) vertices in $\mcl P\setminus \{P\}$ which are not already joined by edges to $P$, so that each vertex of $G_1$ has degree at least 2. Note that $G_1$ still has at most $2M$ edges. 

Since each vertex of $G_1$ is joined to at least 2 other vertices, each connected component of $G_1$ has at least three vertices. Hence there are at most $M/3$ connected components of $G_1$. Consequently, we can produce a connected graph $G$ by adding at most $M/3 $ edges to $G_1$ to link up the connected components. 
The graph $G $ has $M$ vertices and at most $7M/3 $ edges. By Lemma~\ref{lem-independent-set}, there is a subset $\wt{\mcl P}\subset\mcl P$ such that the bound~\eqref{eqn-independent-set} holds and no two geodesics in $\wt{\mcl P}$ are joined by an edge in $G$.
Hence also no two edges in $\wt{\mcl P}$ are joined by an edge in $G_0$, i.e., for each $P\in\wt{\mcl P}$, $P$ is the only geodesic in $\wt{\mcl P}$ which hits $u_P$. 
\end{proof}

\subsection{Dimension bound for points joined to zero by two geodesics}
\label{sec-two-geo-dim}

The following proposition is the remaining statement needed to conclude the proof of Theorem~\ref{thm-zero-geo}.

\begin{prop} \label{prop-two-geo-dim}
Let $\mcl R_{\geq 2}$ be the set of points $z\in\BB C$ for which there are at least two $D_h$-geodesics from 0 to $z$, as in Proposition~\ref{prop-three-geo-dense}. Almost surely, $\dim_{\mcl H}^\gamma \mcl R_{\geq 2} \leq d_\gamma-1$. 
\end{prop}

The proof of Proposition~\ref{prop-two-geo-dim} is essentially the same as the proof of Proposition~\ref{prop-multi-geo-dim} with $n=1$. 
In fact, we will re-use most of the estimates which go into the proof of Proposition~\ref{prop-multi-geo-dim}. 
Let us, therefore, assume that we are in the setting discussed just after Proposition~\ref{prop-multi-geo-dim} with $n=1$, so that $U_1,U_1',V_1'\subset\BB C$ are bounded open sets which lie at positive distance from $\bdy\BB D$ and satisfy $\ol U_1\subset U_1'$ and $\ol U_1' \subset V_1$. 
We also assume that $V_1$ lies at positive distance from 0.

Let $\wt{\mcl R}_{\geq 2}$ be the set of $z\in\BB C$ for which there is a $D_h$-geodesic $P^1$ from 0 to $z$ which enters $U_1$ and a geodesic $P^0$ from 0 to $z$ which is disjoint from $V_1$. Exactly as in the discussion following Proposition~\ref{prop-multi-geo-dim}, it suffices to show that a.s.\
\eqb \label{eqn-two-geo-subset}
\dim_{\mcl H}^\gamma \wt{\mcl R}_{\geq 2} \leq d_\gamma-1  . 
\eqe

Define $G(C)$ for $C>0$ as in~\eqref{eqn-geo-count-reg} and $E^\ep(0,w)$ for $w\in\BB C$ and $\ep > 0$ as in Section~\ref{sec-geo-count-prob} with $z=0$.

\begin{lem} \label{lem-random-pt'}
Let $W\subset\BB C$ be a bounded open set which lies at positive distance from $\ol V_1$. 
Conditional on $h$, let $\BB w$ be sampled from $ \mu_h|_W$, normalized to be a probability measure.
For each $C>1$ and each $\ep\in (0,1)$, 
\eqb \label{eqn-random-pt'}
\BB P\left[ E^\ep(0,\BB w) \cap G(C)  \right] \leq \ep^{1 + o_\ep(1)} 
\eqe
where the rate of the $o_\ep(1)$ is deterministic.
\end{lem}
\begin{proof}
Write $\wt{\BB P}$ for the law of $(h,\BB w)$ weighted by $ \mu_h(W)$, normalized to be a probability measure. 
By exactly the same proof as in Lemma~\ref{lem-two-point-law}, under $\wt{\BB P}$ the conditional law of $h$ given $\BB z$ is the same as the $\BB P$-law of the field $h - \gamma \log|\cdot-\BB z| $, normalized to have average zero over $\bdy\BB D$. 
We now apply Lemma~\ref{lem-geo-count-prob-f} and follow exactly the same argument as in the proof of Lemma~\ref{lem-random-pt} to obtain~\eqref{eqn-random-pt'}. 
\end{proof}

\begin{lem} \label{lem-subset-cover'}
Let $W\subset\BB C$ be a bounded open set which lies at positive distance from $\ol V_1$ and let $K\subset W$ be compact. 
With probability tending to 1 as $\ep\rta 0$, the set $\wt{\mcl R}_{\geq 2}  \cap K$ can be covered by $\ep^{ - (d_\gamma - 1) +o_\ep(1)}$ sets of the form $\mcl B_\ep(w;D_h)$ for $w\in W$. 
\end{lem}
\begin{proof}
This follows from exactly the same proof as in Lemma~\ref{lem-subset-cover}, but with Lemma~\ref{lem-random-pt'} used in place of Lemma~\ref{lem-random-pt}. 
\end{proof}

\begin{proof}[Proof of Proposition~\ref{prop-two-geo-dim}]
The relation~\eqref{eqn-two-geo-subset} follows from exactly the same argument used to prove Proposition~\ref{prop-subset-dim}, with Lemma~\ref{lem-subset-cover'} used in place of Lemma~\ref{lem-subset-cover}. 
\end{proof}

\appendix

\section{Appendix}

Here we collect some basic technical lemmas which are used at various places in the paper. The proofs of these lemmas do not use any other results in the paper. Our first lemma is a purely topological statement which is used in the proof of Lemma~\ref{lem-three-geo-dense-pt}. 

\begin{lem} \label{lem-connected-bdy}
Let $X$ be a path connected, locally path connected topological space whose first homology group $H_1(X)$ (with integer coefficients) is trivial. 
Let $K\subset X$ be a closed connected set. 
If $X\setminus K$ is connected, then $\bdy K$ is connected.
\end{lem}
\begin{proof}
Our argument follows closely an argument given by Moishe Kohan on Math stack exchange, see~\cite{kohan-mse-post}.

Assume by way of contradiction that $\bdy K$ is not connected.
Then there exist disjoint open sets $A,B\subset X$ such that $\bdy K\subset A\cup B$ and each of $A\cap \bdy K$ and $B\cap \bdy K$ is non-empty.
We can assume without loss of generality that each connected component of $A$ and each connected component of $B$ has non-empty intersection with $\bdy K$ (otherwise, we can remove some components).

We observe that the sets $U := K \cup A \cup B$ and $V := (X\setminus K) \cup A \cup B$ are each open and connected.
Indeed, $V$ is the union of three open sets by definition and $U$ is the union of the interior of $K$ and the open sets $A$ and $B$. 
The set $U$ is the union of the connected set $K$ and the connected components of $A$ and $B$, each of which intersect $K$, so $U$ is connected.
Similarly, $V$ is the union of the connected set $X\setminus K$ and the connected components of $A$ and $B$, each of which intersect $\bdy K$ and hence also $X\setminus K$. 
So, $V$ is also connected. On the other hand, the set $U\cap V = A\cup B$ is not connected. 

We briefly recall that the \emph{reduced homology} groups of a topological space $Y$. 
For $n\geq 1$, the $n$th reduced homology group is the same as the ordinary $n$th homology group $H_n(Y)$.
For $n=0$, the $0$th reduced homology group $\wt H_0(Y)$ is defined so that $H_0(Y) = \wt H_0(Y) \oplus \BB Z$, which makes it so that $\wt H_0(Y) \cong \BB Z^{m-1}$, where $m$ is the number of connected components of $Y$. 

We now apply the Mayer-Vietoris sequence for reduced homology, which implies in particular that if $X = U\cup V$ is the union of open sets with non-empty intersection, then we have an exact sequence (i.e., the kernel of each map is the image of the previous map)
\eqb
H_1(X) \rta \wt H_0(U\cap V) \rta \wt H_0(U) \oplus \wt H_0(V)  .
\eqe
In our setting, $H_1(X) = 0$ by assumption and $\wt H_0(U) = \wt H_0(V) = 0$ since $U$ and $V$ are each connected. 
Therefore, the exactness of the sequence implies that $\wt H_0(U\cap V) = 0$, i.e., $U\cap V = A\cup B$ is connected. This contradicts the earlier statement that $A\cup B$ is not connected, so we conclude that $\bdy K$ must be connected.
\end{proof}

The following lemma is a variant of a standard fact about Gaussian multiplicative chaos measures which dates back to Kahane~\cite{kahane}.

\begin{lem} \label{lem-two-point-law}
Let $h$ be a whole-plane GFF normalized so that its average over $\bdy\BB D$ is zero.
Also let $Z,W \subset \BB C$ be bounded open sets which lie at positive distance from each other.
The following two probability measures agree (here, all laws are assumed to be normalized to be probability measures). 
\begin{enumerate}[(i)]
\item The law of the triple $(\wt h ,z,w)$ where $\wt h$ is sampled from the law of $h$ weighted by $\mu_h(Z)\mu_h(W)$, and, conditional on $h$, $z$ and $w$ are sampled independently from $\mu_h|_Z$ and $\mu_h|_W$, respectively. \label{item-law-quantum} 
\item The law of the triple $(h - \gamma \log|\cdot-z| - \gamma \log |\cdot-w| - \gamma c_z - \gamma c_w ,z,w)$ where $h$ is a whole-plane GFF as above, \label{item-law-lebesgue}
\eqbn
c_z := \frac{1}{2\pi} \int_{\bdy\BB D} \log |u-z|^{-1} \,du ,
\eqen
and the pair $(z,w)$ is independent from $h$ and is sampled from Lebesgue measure on $Z\times W$ weighted by 
\eqb \label{eqn-two-pt-weight}
|z-w|^{-\gamma^2} (|z|_+  |w|_+)^{2\gamma^2} ,\quad\text{where} \quad |u|_+ := \max\{|u| , 1\} .
\eqe 
\end{enumerate}
\end{lem}
\begin{proof}
By a slight abuse of notation, we write $dh$ for the law of $h$, $dz$ for Lebesgue measure on $Z$, and $dv$ for Lebesgue measure on $W$. 
For $\ep > 0$, define the probability measure 
\eqb
\Theta^\ep := \mcl Z_\ep^{-1} e^{\gamma ( h_\ep(z) + h_\ep(w)) }  \,dh \,dz \, dw ,
\eqe
where $\mcl Z_\ep $ is a normalizing constant. We will show that as $\ep\rta 0$, the law $\Theta^\ep$ converges to each of~\eqref{item-law-quantum} and~\eqref{item-law-lebesgue}. 

By integrating out $z$ and $w$, we see that the marginal law of $h$ under $\Theta^\ep$ converges as $\ep\rta 0$ to the measure $\mu_h(Z) \mu_h(W) \,dh$, normalized to be a probability measure. 
The $\Theta^\ep$-conditional law of $(z,w)$ given $h$ is equal to $\mcl Z_\ep^{-1} e^{\gamma (h_\ep(z) + h_\ep(w))} \,dz \,dw$, which converges as $\ep\rta 0$ to $\mu_h|_Z \times \mu_h|_W$, normalized to be a probability measure. 
Therefore, as $\ep\rta 0$ the probability measure $\Theta^\ep$ converges to the law~\eqref{item-law-quantum}.

On the other hand, the covariance function of $h$ is given by
\eqbn
\op{Cov}\left( h(z) , h(w) \right)  = \log\frac{|z|_+ |w|_+}{|z-w|} ;
\eqen
see, e.g.,~\cite[Section 2.1.1]{vargas-dozz-notes}. 
By integrating this covariance function against the uniform measure on $\bdy B_\ep(z)\times\bdy B_\ep(w)$, we get that if $\ep  < |z-w|$, then
\eqb
\op{Cov}\left( h_\ep(z) , h_\ep(w) \right)  = \log\frac{|z|_+ |w|_+}{|z-w|} + o_\ep(1) 
\eqe
where the $o_\ep(1)$ tends to zero as $\ep\rta 0$ uniformly on compact subsets of $\BB C$. 
Similarly,
\eqbn
\op{Var} h_\ep(z) = \log \ep^{-1} + 2 \log |z|_+ + o_\ep(1) . 
\eqen
Therefore,
\eqb
\BB E\left[ e^{\gamma ( h_\ep(z) + h_\ep(w)) } \right] 
=  e^{o_\ep(1)} \ep^{-\gamma^2}   |z-w|^{-\gamma^2} (|z|_+  |w|_+)^{2\gamma^2}   .
\eqe
So, under $\Theta^\ep$, the marginal law of $(z,w)$ is given by Lebesgue measure on $Z\times W$ weighted by $e^{o_\ep(1)} |z-w|^{-\gamma^2} (|z|_+  |w|_+)^{2\gamma^2} $, normalized to be a probability measure. 

The $\Theta^\ep$-conditional law of $h$ given $(z,w)$ is the same as the marginal law of $h$ weighted by $e^{\gamma(h_\ep(z) + h_\ep(w))} $ (normalized to be a probability measure). 
To describe this conditional law, let 
\eqb
\phi_\ep^z(u) := -\log\max\{\ep , |u-z|\} - c_z^\ep ,
\eqe
where $c_z^\ep$ is the average of $-\log\max\{\ep , |u-z|\}$ over $\bdy\BB D$. 
Then by definition we have $h_\ep(z) = (h , \phi_\ep^z)_\nabla$. Consequently, the law of $h$ weighted by $e^{\gamma(h_\ep(z) + h_\ep(w))} $ is the same as the law of 
\eqbn
h + \gamma \phi_\ep^z + \gamma \phi_\ep^w  .
\eqen
Since $\phi_\ep^z \rta -\gamma\log|\cdot-z| - c_z$ in the distributional sense as $\ep\rta 0$ and similarly for $\phi_\ep^w$, we get that $\Theta^\ep$ converges to the law~\eqref{item-law-lebesgue} as $\ep\rta 0$. 
\end{proof}

Our next lemma is essentially a version of the statement that the Minkowski dimension of the $\gamma$-LQG metric is $d_\gamma$, which was proven in~\cite{afs-metric-ball}.

\begin{lem} \label{lem-ball-cover}
Let $W\subset\BB C$ be a bounded open set, let $K \subset W$ be a closed set, and fix a small constant $\zeta \in (0,1)$. 
For $\ep > 0$, let $M_\ep := \lfloor \ep^{-d_\gamma-\zeta} \rfloor$. Conditional on $h$ let $\{\BB z_k\}_{k\in [1,M_\ep]_{\BB Z}}$ be conditionally i.i.d.\ samples from $\mu_h|_W$, normalized to be a probability measure.
With probability tending to 1 as $\ep\rta 0$,  
\eqb \label{eqn-ball-cover}
K \subset \bigcup_{k=1}^{K_\ep} \mcl B_\ep(\BB z_k ; D_h) \subset B_{\ep^{1/\chi} }(W)  
\eqe
where $\chi  := \xi(Q+2) + \zeta$.
\end{lem}
\begin{proof}
By the local bi-H\"older continuity of $D_h$ and the Euclidean metric~\cite[Theorem 1.7]{lqg-metric-estimates}, it holds with probability tending to 1 as $\ep\rta 0$ that
\eqb \label{eqn-ball-cover-holder} 
|z-w|^{\xi(Q+2)+\zeta} \leq D_h(z,w) \leq |z-w|^{\xi(Q-2)-\zeta} ,\quad\forall z,w\in B_1(W) . 
\eqe
From this, the second inclusion in~\eqref{eqn-ball-cover} is immediate. 

To prove the first inclusion in~\eqref{eqn-ball-cover}, we set $\chi' := \xi(Q-2)-\zeta$ and choose for each $\ep > 0$ a deterministic collection $\mcl W_\ep$ of $O_\ep(\ep^{-2/\chi' })$ points in $W$ such that 
\eqbn
K \subset\bigcup_{w \in \mcl W_\ep} B_{(\ep/2)^{1/\chi'}/2}(w) .
\eqen
By~\eqref{eqn-ball-cover-holder} we have $B_{(\ep/2)^{1/\chi'}}(w)\subset \mcl B_{\ep/2}(w;D_h)$ for each $w\in \mcl W_\ep$.
Hence with probability tending to 1 as $\ep\rta 0$, the balls $\mcl B_{\ep/2}(w;D_h)$ for $w\in\mcl W_\ep$ cover $K$. 

By~\cite[Theorem 1.1]{afs-metric-ball}, it holds with probability tending to one as $\ep\rta 0$ that
\eqb
\mu_h(\mcl B_{\ep/2} (w;D_h)) \geq \ep^{d_\gamma + \zeta/2} , \quad\forall w \in W .
\eqe
Hence, it holds with probability tending to one as $\ep\rta 0$ that for each $k\in [1,M_\ep]_{\BB Z}$ and each $w\in\mcl W_\ep$,
\eqb
\BB P\left[ \BB z_k \in \mcl B_{\ep/2}(w;D_h) \,|\, h \right] \geq \frac{\ep^{d_\gamma+\zeta/2}}{\mu_h(W)} .
\eqe
Since the $\BB z_k$'s are conditionally independent given $h$ and $M_\ep = \lfloor \ep^{-d_\gamma-\zeta}\rfloor$, this implies that for each $w\in\mcl W_\ep$,
\allb \label{eqn-ball-cover-conc}
\BB P\left[\text{$\BB z_k \in \mcl B_{\ep/2}(w;D_h) $ for at least one $k \in [1,M_\ep]_{\BB Z}$} \,|\, h \right] 
&\geq 1 - \left( 1 - \frac{\ep^{d_\gamma+\zeta/2}}{\mu_h(W)} \right)^{M_\ep} \notag\\
&\geq 1 - O_\ep\left( \exp\left( - \frac{1}{\ep^{\zeta/2} \mu_h(W)}   \right)  \right) .
\alle
We take a union bound over all $O_\ep(\ep^{-2/\chi'})$ points $w\in \mcl W_\ep$ to get that with probability tending to 1 as $\ep\rta 0$, each of the balls $\mcl B_{\ep/2}(w;D_h)$ for $w\in \mcl W_\ep$ contains at least one of the $\BB z_k$'s. For such a choice of $k$, we have $\mcl B_{\ep/2}(w;D_h) \subset \mcl B_\ep(\BB z_k ; D_h)$. Since the balls $\mcl B_{\ep/2}(w;D_h)$ for $w\in\mcl W_\ep$ cover $K$ with probability tending to one as $\ep\rta 0$, we deduce that the balls $\mcl B_\ep(\BB z_k ; D_h)$ also cover $K$ with probability tending to one as $\ep\rta 0$. 
\end{proof}

\bibliography{cibib}
\bibliographystyle{hmralphaabbrv}

\end{document}